\documentclass[12pt]{article}
\usepackage[tight,footnotesize]{subfigure}

\usepackage{multicol}

\usepackage{enumerate}
\usepackage{url}
\usepackage{hyperref}
\usepackage{color}
\usepackage{amsmath,amssymb}

\usepackage{amsthm}
\usepackage{algorithm}
\newtheorem{theorem}{Theorem}
\newtheorem{lemma}{Lemma}
\newtheorem{proposition}{Proposition}
\newtheorem{definition}{Definition}

\newtheorem{assumption}{Assumption}

\usepackage{mathrsfs}
\usepackage{cite}
\usepackage{caption}
\usepackage{graphicx}
\usepackage{array}
\usepackage{tabularx}
\usepackage{algorithmic}

\usepackage{pgfplots}
\pgfplotsset{compat=1.15}
\usepackage{mathrsfs}
\usetikzlibrary{arrows}
\usepackage{placeins}

\newcounter{algsubstate}

\def\algbackskip{\hskip-\ALG@thistlm}
\makeatother
\newcommand{\paren}[1]{\left(#1\right)}
\renewcommand{\equiv}{:=}

\newcommand{\xbar}{{\overline{x}}}

\newcommand{\norm}[1]{\left\|#1\right\|}

\newcommand{\R}{\mathbb{R}} 

\newcommand{\Ebb}{\mathbb{E}}
\newcommand{\M}{\mathscr{M} }
\newcommand{\LF}{\mathscr{F}}
\newcommand{\SUPP}{\mathscr{Z}}
\newcommand{\SR}{\mathscr{S}}
\newcommand{\SYM}{\mathscr{X}}

\newcommand{\dist}{\mathrm{dist}}

\newcommand{\Id}{\mathrm{Id}}
\newcommand{\DR}{\mathrm{DR}}

\newcommand{\gph}{\mathrm{gph} }

\newcommand{\Fix}{\mathrm{Fix}\ }
\newcommand{\sym}{\mathrm{sym}}
\newcommand{\conv}{\mathrm{conv}}
\newcommand{\pncone}[1]{N^{\text{\rm prox}}_{#1}} 

\newcommand{\Tcal}{\mathcal{T}}

\newcommand{\Dhat}{{\widehat{D}}}
\newcommand{\Nhat}{{\widehat{N}}}
\newcommand{\xhat}{{\widehat{x}}}
\newcommand{\yhat}{{\widehat{y}}}
\newcommand{\zhat}{{\widehat{z}}}
\newcommand{\uhat}{{\widehat{u}}}

\newcommand{\Ibbhat}{{\widehat{\mathbb{I}}}}

\newcommand{\qedbox}{\hfill $\Box$ \\ \medskip}
\newcommand{\email}[1]{\texttt{#1}}

\newcommand{\footremember}[2]{%
    \footnote{#2}
    \newcounter{#1}
    \setcounter{#1}{\value{footnote}}%
}
\newcommand{\footrecall}[1]{%
    \footnotemark[\value{#1}]%
}

\DeclareMathOperator{\argmin}{argmin}
\def\sparse{\mathcal{A}_s}

\title{Algorithmic approaches to avoiding bad local minima in nonconvex  inconsistent feasibility}
\author{
Thi Lan Dinh\footremember{1}{Institut f\"ur Numerische und Angewandte Mathematik, Georg-August-Universit\"at G\"ottingen, G\"ottingen, Germany.(\email{t.dinh@math.uni-goettingen.de}, \email{r.luke@math.uni-goettingen.de})}
\and 
Wiebke Bennecke\footremember{2}{I. Physikalisches Institut, Georg-August-Universität Göttingen, Göttingen, Germany,  G\"ottingen, Germany.
(\email{wiebke.bennecke@stud.uni-goettingen.de}, \email{gsmjansen@uni-goettingen.de}, \email{smathias@uni-goettingen.de})}
\and G. S. Matthijs Jansen\footrecall{2} 
\and D. Russell Luke\footrecall{1}
\and Stefan Mathias\footrecall{2}}

\date{\today}
\begin{document}
\maketitle

\begin{abstract}
We report on the use of algorithms to avoid or move away from ``bad'' local minima in nonconvex optimization.  Our study is phenomenological and
empirical, focusing on the performance of cyclic projections, the cyclic relaxed Douglas-Rachford algorithm, and relaxed Douglas-Rachford splitting on the product space
for orbital tomographic imaging from angle-resolved photon emission spectroscopy (ARPES) measurements, with both synthetic and laboratory data.  Cyclic projections and
Douglas-Rachford on the product space are both well-known methods, but cyclic relaxed Douglas-Rachford was only recently fully characterized in a companion paper to the
present study.  Only one other study of note has investigated the performance of all three of these algorithms for inconsistent nonconvex feasibility.  We show that the relaxed
Douglas-Rachford algorithm on the product space, while exhibiting very poor convergence rates, can be used to filter out bad local minima from all cyclic algorithms.
Our numerical experiments lead to the following recommendation: run cyclic projections to find some fixed point, and from this fixed point run relaxed Douglas-Rachford algorithm
on the product space with as large a relaxation parameter as is numerically stable in order to escape poor local minima.  This advice runs counter to the current practice for
phase retrieval, where a Douglas-Rachford-type algorithm is run for several iterations, and then cyclic projections is used to ``clean up'' the images.
\end{abstract}




\section{Introduction}
In nonconvex optimization, as with finding the zeros of a nonlinear function, there are not many tools that can be used to avoid undesirable local minima.
One of the few available tools is, however, something that is often avoided or ignored in the theory of split feasibility, namely inconsistency.
The relaxed Douglas-Rachford algorithm, first proposed in \cite{Luke2005}, leverages inconsistency for solving phase retrieval problems. The only
nonconvex analysis of this algorithm for the inconsistent case that we are aware of was developed in \cite{Luke08, luke2020convergence}.
The cyclic Douglas-Rachford method, first proposed in \cite{BorTam14}, also showed some potential for inconsistent feasibility and
extended the Douglas-Rachford algorithm to many-set (more than two) feasibility problems.
This was extended to the cyclic relaxed Douglas-Rachford algorithm in \cite{luke2018relaxed} and this variant showed some promise for
inconsistent nonconvex feasibility in numerical studies conducted in \cite{luke2019optimization}.   We have recently characterized the fixed points
of the cyclic relaxed Douglas-Rachford algorithm in \cite{DinhJansenLuke2024CDRl}, but the
potential of this algorithm for filtering out bad local minima was not tested until this numerical study.
The standard practice in the optics community is to run a
few iterations of a Douglas-Rachford-type algorithm, and then to run several iterations of cyclic projections to arrive at a satisfactory fixed point \cite{Luke17}.
We show in this paper that this practice, while not unreasonable, has things the wrong way around:  one should start with cyclic projections and then run a relaxed
Douglas-Rachford algorithm in order to move out of undesirable local domains of attraction.

Relaxations of the Douglas-Rachford algorithm have been extensively studied primarily with the goal of stabilizing the original Douglas-Rachford mapping
\cite{Luke08, hesse2013nonconvex, Phan16, BorTam14, thao2018convergent} but also with an eye toward obtaining/preserving linear convergence for nonconvex
feasibility \cite{HesseLukeNeumann14, LiPong16, DaoPhan18, luke2020convergence}.  Beyond not converging
for inconsistent feasibility problems, another drawback of the Douglas-Rachford algorithm is that there is not a unique way to extend it to splitting with more than two operators.
For multi-set feasibility,  one possibility is to formulate the problem as two-set feasibility between the sets on the product space and the {\em diagonal}
of the product space \cite{Pierra84}.  Another approach, first studied in \cite{BorTam14}, is to apply the two-operator Douglas-Rachford mappings pairwise in a cyclic manner.
Obviously these two different approaches lead to different algorithms with different convergence properties.  A cyclic relaxed Douglas-Rachford method
was first proposed in \cite{luke2018relaxed} that combined the idea of \cite{BorTam14} with the relaxation in \cite{Luke08}.  This algorithm has recently been fully studied
in the nonconvex, inconsistent setting in \cite{DinhJansenLuke2024CDRl},
where the fixed points are characterized and local convergence is established under the weakest assumptions to date.

In this work we present numerical results showing the relative merits of all of the above strategies, leading to the observation that the relaxed Douglas-Rachford algorithm is a
reasonable candidate for avoiding bad fixed points of the cyclic projections algorithm.  This observation follows from the characterization of the fixed points in
\cite[Theorem 3.1]{DinhJansenLuke2024CDRl} and \cite[Theorem 3.13]{luke2020convergence},
and leads to the conclusion that the product-space, relaxed Douglas-Rachford algorithm can be used to move away from bad local minima of the cyclic projections algorithm.
Our numerical demonstration focuses on the problem of tomographic reconstruction of molecular electronic structures from  angle-resolved photoemission spectroscopy (ARPES)
data studied in \cite{Dinh_2024}.  We refer to this as the photoemission orbital tomography problem.  While this problem is quite specific, it has features that are common
in many applications, and therefore serves as a good test case.  All of these algorithms have been compared on a number of different related problem classes
\cite{luke2019optimization}, but our results for the orbital tomography problem provide a more complete picture than the conclusions resulting from \cite{luke2019optimization}.

In Section \ref{sec:prelimiaries}, we recall the mathematical definitions necessary for the convergence statements in this study.
In Section \ref{sec:3D.problem}, we review the physical ARPES orbital tomography problem, reformulate this as a feasibility problem,
and  outline the numerical methods used for its solution (Section \ref{sec:algorithms}).
Since all of the algorithms in this comparison are built from projection operators,
we derive the formulas for the projectors onto the sets in the ARPES orbital tomography problem in Section \ref{sec:projection.formulas}, and
verify the main assumption for local linear convergence of the three algorithms for the
ARPES orbital tomography problem, with explicit formulations of the parameters involved in the rate constant (Section \ref{sec:convergence.guarantee}).
In Section \ref{sec:experiments}, we present numerical results using the dataset from \cite{Dinh_2024}. Compared to the current state-of-the-art cyclic
projection method, CDR$\lambda$ achieves a higher success rate. Additionally, DR$\lambda$ is shown to steer {\em away} from bad local minima of the other
algorithms, although this is at the cost of very slow convergence rates.
The code and data for the laboratory data is available from \cite{Bennecke25_data}.

\section{Preliminaries}\label{sec:prelimiaries}
Throughout this paper, $\Ebb$ is a finite-dimensional Euclidean space. The Euclidean space of interest in particular is $\Ebb=\mathbb C^N$, complex-valued
real numbers with the real inner product.
We recall definitions of almost nonexpansive/almost $\alpha$-firmly nonexpensive (see \cite[Definition 1]{berdellima2022alpha}).
Let $D$ be a nonempty subset of $\mathbb{E}$ and let $T$ be a (set-valued) mapping from $D$ to $\mathbb{E}$, i.e., $T:D\rightrightarrows \Ebb$.
\begin{enumerate}[(a)]
   \item  $T$ is said to be  {\em pointwise almost nonexpansive on $D$ at $ y \in D$} if there exists
        a constant $\epsilon\in[0,1)$ such that 
\begin{eqnarray}\label{e:epsqnonexp}
&&\|x^+- y^+\|\leq\sqrt{1+\epsilon}\|x- y\|\,,\\
&&~\forall~ y^+\in T y \mbox{ and } \forall~ x^+\in Tx \mbox{ whenever }x\in D. \nonumber
\end{eqnarray}
\item 
$T$ is {\em pointwise almost $\alpha$-firmly nonexpensive} (abbreviated pointwise a$\alpha$-fne) at $y\in D$
whenever
\begin{equation}\label{eq:aafne}
\begin{aligned}
&\exists \,\alpha\in (0,1) \mbox{ and }\epsilon\in [0,1): \forall x\in D\,,\,\forall x^+\in Tx\,,\,\forall y^+\in Ty\,, \\
     &\qquad \|x^+-y^+\|^2\leq (1+\epsilon)\|x-y\|^2-\frac{1-\alpha}{\alpha}\psi(x,y,x^+,y^+)\,,
     \end{aligned}
\end{equation}
where
\begin{equation}
    \psi(x,y,x^+,y^+):=\|(x^+-x)-(y^+-y)\|^2.
\end{equation}
If a mapping is pointwise a$\alpha$-fne at a point $y$, then it is also pointwise almost nonexpansive there.
When $\epsilon=0$ in
\eqref{e:epsqnonexp} and \eqref{eq:aafne} above,
the qualifier ``almost'' is dropped and mappings satisfying \eqref{eq:aafne} are abbreviated (pointwise) $\alpha$-fne.
In the present setting with $\epsilon=0$, what we are calling $\alpha$-fne mappings are more commonly known as {\em averaged} mappings \cite{BaiBruRei78},
but the form given in \eqref{eq:aafne} generalizes to metric spaces where the term ``averaged'' makes no sense and the four point mapping
$\psi(x,y,x^+,y^+)$ (called the {\em transport discrepancy} in \cite{berdellima2022alpha})
is of central importance.
\end{enumerate}

The {\em projector} of a point $x$ onto the set $A\subset\Ebb$, is the mapping of $x$ to the set of
points where the distance to the set $A$ is attained, 
\begin{equation*}
    P_A x := \argmin_{a\in A} \|a - x\|.
\end{equation*}
When $A$ is {\em closed and non-empty} this set is nonempty. 
If $y \in  P_A x$, then $y$ is called a projection of $x$ on $A$. 
The reflector is defined by $R_ C:= 2P _C - \Id$. If $y\in R_Cx$, then $y$ is called reflection of $x$ across $C$.
The regularity of closed sets is categorized in terms of the a$\alpha$-fne regularity of the projection mapping for each set.

Given set-value mapping $T:\Ebb \rightrightarrows \Ebb$. Inverse operator $T^{-1}$ is defined by
$    T^{-1} (a) := \{x \in \Ebb \ | \ a \in T x \}
$
and the graph of the mapping $T$ is the set
$
    \gph (T):=\{(x,y):x\in \Ebb\,,\,y\in Tx\}
$
which is a subset of $\Ebb\times \Ebb$.

For a subset $A\subset\Ebb$, the image of $A$ under $T$, denoted by $TA$, is defined by
\begin{equation*}
    T A:=\bigcup_{a\in A} Ta\,.
\end{equation*}
When the mapping $T:\Ebb \rightrightarrows \Ebb$ is single-valued at $y\in \Ebb$, that is if $Ty = \{z\}$ is a singleton, then
we simply write $z=Ty$.
The mapping $T$ is single-valued on $U\subset E$ if it is single-valued at all points $y\in U$.
Compositions of set-valued mappings $T_1, T_2:\Ebb \rightrightarrows \Ebb$ are denoted $T_2T_1$ or $T_2\circ T_1$ from $\Ebb$ to $\Ebb$ and
\begin{equation}
    T_2T_1x= (T_2\circ T_1)x:= \bigcup_{y\in T_1x} T_2y.
\end{equation}
The set of fixed points of $T$ is defined by
$    \Fix T:=\{x\in \Ebb\,|\, x\in Tx\}.
$
It is easy to see that when $T$ is pointwise almost nonexpansive at $y\in \Fix T$, then $T$ is single-valued there.

\section{The 3D Photoemission orbital tomography problem}\label{sec:3D.problem}
\subsection{Problem statement}\label{sec:problem.statement}
We now specialize the theory in \cite{DinhJansenLuke2024CDRl} to the very concrete problem of orbital tomography with angle-resolved photoemission spectroscopy (ARPES) data.
The physical domain of the objects which are to be recovered from ARPES data is represented by $D\subset \mathbb{R}^3$.
A point in $D$ corresponding to an index $(i,j,l)$ is denoted $(x_i, y_j, z_l)$, where
$(i,j,l)\in \mathbb{I}\equiv \{1,\dots,N_x\}\times \{1,\dots,N_y\}\times \{1,\dots,N_z\}$. Here, $N_*$ represents the number of voxels in the $x$, $y$, and $z$ directions,
respectively. The total number of voxels $|\mathbb{I}|$ is denoted by $N=N_x\times N_y\times N_z$.
The ``object" in question is the spatial representation of a {\em molecular orbital}, modeled for convenience as a complex-valued vector
$u\in \mathbb{C}^N$, each element of which, $u_{(i,j,l)}$, representing the ``value" (yielding the probability density \cite{Born}) of the
electronic orbital at the points $(x_i, y_j, z_l)\in D$.  The only complex values that the orbital $u_{(i,j,l)}$ takes are either
$u_{(i,j,l)}=+|u_{(i,j,l)}|$ or $u_{(i,j,l)}=-|u_{(i,j,l)}|$ corresponding to the {\em charge} of the electron orbital.

The mathematical model for 3D photoemission orbital tomography is described in detail in \cite{Dinh_2024}.  Since our primary interest is the mathematical
structure of the problem, we will not derive the mathematical model for the data.  The data domain (image space of the model),
where the ARPES data is observed, is also a three dimensional real space
denoted by $\Dhat\subset \mathbb{R}^3$.
This is discretized in a similar manner to the discretization of the domain, but potentially with different numbers of pixels/voxels.
The point in $\Dhat$ corresponding to the index $(i,j,l)\in \Ibbhat\equiv \{1,\dots,\Nhat_x\}\times \{1,\dots,\Nhat_y\}\times \{1,\dots,\Nhat_z\}$
is denoted  $(\xhat_i,\yhat_j,\zhat_l)$. (Physicists reading this are more accustomed to seeing the symbol $k_*$ for the momentum space coordinates.)
Here, $\Nhat_*$ represents the number of voxels in the $\xhat$, $\yhat$, and $\zhat$ directions in the image space.
The total number of voxels $|\Ibbhat|$ is denoted by $\Nhat=\Nhat_x\times \Nhat_y\times \Nhat_z$.
In principle, the image $\uhat$ of the object $u$ to be recovered,
is also a complex-valued vector  $\uhat\in \mathbb{C}^\Nhat$, each element of which, $\uhat_{(i,j,l)}$, representing the ``value" of the  electronic
orbital upon transformation at the points $(\xhat_i, \yhat_j, \zhat_l)\in \Dhat$ in the image domain.  The transformation operator of the molecular
orbital $u$ from the physical domain $D$ to the observable $\uhat$ in the data domain $\Dhat$ is, to first approximation, a Fourier transform, $\uhat= Fu$.

The physical data are pointwise measurements of the {\em amplitude of $\uhat$} on {\em spheres centered at the origin} in $\Dhat$.
The radius of a measurement sphere is determined by the energy at which the object is illuminated.
Upon discretization, the data, denoted $b_{(i,j,l)}\in \mathbb{R}_+$, specifies the amplitudes $|\uhat_{(i,j,l)}|$ only for indexes
$(i,j,l)\in \mathbb{S}\subset\Ibbhat$.
The set of all possible vectors $u$ satisfying the data measurements is defined by
\begin{equation}
\M := \{u\in {{\mathbb{C}}^{N}}\mid \ |\uhat_{(i,j,l)}|=b_{(i,j,l)}, \  (i,j,l)\in \mathbb{S}\}\,.
\label{eq:phase_reformulated}
\end{equation}
Again, for the physicists, a particular data set will consist of measurements at several (on the order of $10$, but not less than $4$, see \cite{Dinh_2024})
energies, corresponding to several concentric spheres for each set $\M$\footnote{Actually, the measurements are only {\em hemispheres}
-- see Figure \ref{fig:data.best.reconstruction} -- but due to conjugate symmetry
we can extend this numerically to a complete sphere.}.   The fact that the indexes $\mathbb{S}$ correspond to several spheres in
$\Dhat$ is of no consequence to our present analysis;
for our purposes it is only relevant that $\mathbb{S}$ is a subset of all the indexes $\Ibbhat$.

Regardless of how many energies are used,
the data  set $\M$ will never uniquely determine the electron orbital \cite{Luke17}, so the challenge is to introduce further constraints
that limit the number of {\em local minima} and to develop algorithms that have the {\em best fixed points} from both a mathematical
and physical point of view.   The theoretical study of the CDR$\lambda$ algorithm in \cite[Thoerems 3.1 and 3.2]{DinhJansenLuke2024CDRl}
and  the numerical results reported in Section \ref{sec:experiments} show that the
CDR$\lambda$ algorithm can produce fewer bad local reconstructions than the cyclic projections algorithm used in \cite{Dinh_2024}.

As in \cite{Dinh_2024},  we formulate orbital tomography reconstruction problem as a multi-set feasibility model:
\begin{equation}\label{eq:multiset.model.orbit}
    \text{Find}\ u\in \SYM\cap \SR \cap \SUPP\cap \LF\cap \M,
\end{equation}
where
\begin{itemize}
\item $\M$ is the measurement constraint set defined in \eqref{eq:phase_reformulated}.
\item $\LF$ is the set of points in the object domain  that satisfy a support constraint in the image domain:
\begin{equation}\label{eq:LM}
    \LF: = \{u\in {{\mathbb{C}}^{N}}\mid \ |\uhat_{(i,j,l)}|=0,  \ (\xhat_i,\yhat_j,\zhat_l) \notin B_{\overline{r}}\}\,.
    \end{equation}
    where $B_{\overline{r}}$ is a ball in $\Dhat$.
    \item $\SUPP$ is the set of points that satisfy a support constraint in the object domain:
    \begin{equation}\label{eq:support_constraint}
    \SUPP=\{u\in\mathbb{C}^{N}\mid \ u\odot\Omega=u\},
\end{equation}
for some symmetric binary mask $\Omega\in\{0,1\}^{N}$ and $\odot$ the elementwise Hadamard product.
\item $\SR$ is the set of points that satisfy a {\em sparse real constraint}:
\begin{equation}
\SR:= \{u\in {{\mathbb{C}}^{N}}\mid \ \|u\|_0 \leq s\mbox{ and }\text{Im}(u)=0\},
\label{eq:sparse_set}
\end{equation}
where, $s>0$ is given, Im$(u)$ is the imaginary part of the complex vector $u$, and $\|\cdot\|_0$ denotes the counting function that counts all non-zero elements. 
\item  $\SYM$ is the symmetry and anti-symmetry constraint set:
\begin{equation}\label{eq:symmetry_set}
\begin{aligned}
\SYM :=& \left\{u\in {\mathbb{C}^{N}} \mid u_{(i,j,l)}=u_{(N_x-i+1,j,l)}=-u_{(i,N_y-j+1,l)}=-u_{(i,j,N_z-l+1)}\,,\right.\\
&\qquad\qquad\qquad\qquad\qquad\qquad\qquad\qquad\qquad\qquad\qquad \left. (i,j,l)\in \mathbb{I}
\right\}.
\end{aligned}
\end{equation}
\end{itemize}
The sets $\M$ and $\SR$ are nonconvex, so, without additional assumptions,
convergence to a solution of \eqref{eq:multiset.model.orbit} can only be guaranteed locally.
Note, however, that the sets $\SUPP$ and $\LF$, though convex, do not in general have points in common:  any vector in $\SUPP$ that
is not periodic will not belong to $\LF$.  The feasibility model \eqref{eq:multiset.model.orbit} is
therefore {\em inconsistent} and there do not exist solutions to \eqref{eq:multiset.model.orbit}.  We therefore must be content with
finding points that come as close as possible to satisfying all of the constraints.

\subsection{Projection methods}\label{sec:algorithms}
We compare the performance of three different algorithms, all of which can be written as the
following general fixed point iteration.
\begin{algorithm}
\caption{General fixed point iteration}\label{alg:gen fpa}
\begin{algorithmic}
    \STATE \textbf{Initialization:} Set $\text{tol} > 0$, $n = 0$, $\text{monitor}^{(0)} > \text{tol}$, $u^{(0)} \in \mathbb{C}^N$.
    \WHILE{$\text{monitor}^{(n)} > \text{tol}$}
        \STATE Set $n = n + 1$. Compute
        \begin{equation}\label{eq:1st.CP}
        u^{(n)} \in T u^{(n-1)}
        \end{equation}
        and update $\text{monitor}^{(n)}$.
    \ENDWHILE
\end{algorithmic}
\end{algorithm}
The three algorithms represent different instantiations of the mapping $T$, namely cyclic projections,
$T^{\text{CP}}$, cyclic relaxed Douglas-Rachford $T^{\text{CDR}\lambda}$
and relaxed Douglas-Rachford on the {\em product space}, $T^{\textbf{DR}\lambda}$.
There are a number of termination criteria that could be used.  For most of our experiments, the
exit criterion we use is based on the difference in the {\em shadows} of the iterates $u^{(n)}$ onto the
set $\SYM$, i.e., 
\begin{equation}\label{eq:monitor.CP.CDRl}
\text{monitor}^{(n)}:=\text{monitor}^{(n)}_1:= |P_{\SYM}u^{(n)}-P_{\SYM}u^{(n-1)}|.
\end{equation}  For the cyclic projection mapping $T^{\text{CP}}$, the shadows are just the iterates
themselves.  For the relaxed Douglas-Rachford algorithm below, this will not be the case.

We recall the cyclic projection algorithm in  \cite{Dinh_2024} to numerically solve problem \eqref{eq:multiset.model.orbit}.
Define
\begin{equation}\label{eq:first.CP}
T^{{\text{CP}}}:=P_{\SYM}\circ P_{\SR}\circ P_{\SUPP}\circ P_{\LF}\circ P_{\M}.
\end{equation}
In what follows $T^{\DR\lambda}_{A,B}$ denotes the relaxed Douglas-Rachford operator with respect to the sets $A,B$ for a fixed  relaxation parameter $\lambda\in [0.1]$, i.e.,
\begin{equation}\label{e:T_AB}
    T^{\DR\lambda}_{A,B}=\frac{\lambda}{2}(R_{A }R_{B}+\Id)+(1-\lambda)P_{B}.
\end{equation}
This method has been studied in \cite{luke2018relaxed} where quantitative convergence guarantees and characterization of the fixed points set
were established.  When there are more than two sets, as here, there are a number of different options for applying relaxed Douglas-Rachford-type splitting.  Two options that we compare here are {\em cyclic} relaxed Douglas-Rachford and relaxed Douglas-Rachford on the {\em product space}.  The cyclic version
takes the form
\begin{equation}\label{eq:CDRl_orbit3D}
    T^{\text{CDR$\lambda$}}:=T^{\DR\lambda}_{\SYM,\SR}\circ T^{\DR\lambda}_{\SR,\SUPP}\circ T^{\DR\lambda}_{\SUPP,\LF} \circ  T^{\DR\lambda}_{\LF,\M}  \circ T^{\DR\lambda}_{\M,\SYM}.
\end{equation}
To derive the product space formulation, let
$\mathbf{u}=(u_1,u_2,u_3,u_4,u_5)$, $C=\SYM\times\SR\times\SUPP\times\LF\times\M$ and denote by
$D$ the diagonal set of $\mathbb C^{N^5}$, which is defined by $\{\mathbf{u}=(u,u,u,u,u)\in \mathbb C^{N^5}\}$.
The operator $T^{\text{DR$\lambda$}}$ is given by
\begin{equation}\label{eq:DRl product}
    T^{\textbf{DR$\lambda$}}:= \frac{\lambda}{2}(R_DR_C+\Id)+(1-\lambda)P_C,
\end{equation}
where 
\begin{equation}
    P_C\mathbf{u} = (P_\SYM u_1,P_\SR u_2, P_\SUPP u_3, P_\LF u_4,P_\M u_5),
\end{equation}
and 
\begin{equation}
    [P_D\mathbf{u}]_i = \frac{P_\SYM u_1+P_\SR u_2+ P_\SUPP u_3+ P_\LF u_4+P_\M u_5}{5}, \,\forall i=1,2,3,4,5.
\end{equation}
The number of variables is quintupled, but the computation of projections can be done in parallel.
In the consistent case, this algorithm would be used to solve the problem
\begin{equation}\label{e:feas product space}
    \text{Find } \mathbf{u}\in C\cap D.
\end{equation} In this case, it is easy to see that each entry of a solution $\mathbf{u}$ to \eqref{e:feas product space}
is a point in $\SYM\cap\SR\cap\SUPP\cap\LF\cap\M$.
In the inconsistent case, the fixed points of \eqref{eq:DRl product} are given abstractly by \cite[Theorem 3.13]{luke2020convergence} where
the characterization is in the product space.  The fixed points of $T^{\text{CDR$\lambda$}}$ were characterized
in \cite[Theorems 3.1 and 3.2]{DinhJansenLuke2024CDRl}.

To evaluate the quality of (approximate) fixed points, we calculate the \emph{gap}. The gap at a point is defined as the sum of the distances between the sets at that point.
An (approximate) fixed point that achieves a smaller gap is considered better. Let $ u^{(n)} $ be the  $n$-th iterate of algorithm \ref{alg:gen fpa}.
Set 
\begin{equation}
    \check{u}^{(n)}=\begin{cases}
        u^{(n)}&{\text{ if }T=T^{\text{CP}}}\\
        P_{\SYM}u^{(n)}&{\text{ if }T=T^{\text{CDR$\lambda$}}}\\
        P_{\SYM}u_1^{(n)}&\text{ if $T=T^{\textbf{DR}\lambda}$.}
    \end{cases}
\end{equation}
The gap at $n$-th iteration is computed as follows:
\begin{eqnarray}\label{e:gap}
    \text{gap}^{(n)}:=&\text{gap}(\check{u}^{(n)})\nonumber\\
    :=&\frac{\paren{\text{gap}_{\SYM-\M}+ \text{gap}_{\M-\LF}+\text{gap}_{\LF-\SUPP}+\text{gap}_{\SUPP-\SR}+\text{gap}_{\SR-\SYM}}(\check{u}^{(n)})}{\|b\|},
    \nonumber
\end{eqnarray}
where
\begin{equation*}
    \begin{array}{rcl}
             & \text{gap}_{\SYM-\M}(u) &:= \|u-P_{\M}u\|,\\
         & \text{gap}_{\M-\LF}(u)&:=\|P_{\M}u-P_{\LF}P_{\M}u\|\,, \\
         & \text{gap}_{\LF-\SUPP}(u) &:= \|P_{\LF}P_{\M}u-P_{\SUPP}P_{\LF}P_{\M}u\|\,, \\
         & \text{gap}_{\SUPP-\SR}(u) &:= \|P_{\SUPP}P_{\LF}P_{\M}u - P_{\SR}P_{\SUPP}P_{\LF}P_{\M}u\|\,,\\
         & \text{gap}_{\SR-\SYM}(u) &:=\|P_{\SR}P_{\SUPP}P_{\LF}P_{\M}u-P_{\SYM}P_{\SR}P_{\SUPP}P_{\LF}P_{\M}u\|.
    \end{array}
\end{equation*} 
For the relaxed Douglas-Rachford operator, we also suggest using the gap for monitoring as opposed to the shadows as in \eqref{eq:monitor.CP.CDRl}, i.e. we use
\begin{equation}
    \text{monitor}^{(n)}: = \text{monitor}^{(n)}_2:= |\text{gap}^{(n)}-\text{gap}^{(n-1)}|.
\end{equation}
The projection formulas onto the sets will be computed explicitly in the next section. 

For simulated data, the error between the iterate $u^{(n)}$ and the ground truth is  given by
\begin{equation}\label{eq:error.with.truth}
E^{(n)}= \frac{1}{2}\min\left\{\left\|\frac{u^*}{\|u^*\|}\pm \frac{u^{(n)}}{\|u^{(n)}\|}\right\|
\right\}
\end{equation}
where $u^{(n)}$ denotes the preferred shadow of the reconstructed orbital at the $n$-the iteration and $u^*$ is the ground truth.
The error is the minimum of either the sum or the difference of the reconstruction and the reference ``truth" to account for the unavoidable global phase ambiguity.

\subsection{Projection formulas}\label{sec:projection.formulas}
We now present the explicit formulations for the projectors onto the sets in \eqref{eq:multiset.model.orbit}.
A projection of $u$ onto the set $\M$ with infinite precision arithmetic is computed by \cite[Theorem 4.2]{luke2002optical}
\begin{equation}
  P_{\M}u\equiv \left\{\mathcal{F}^{-1} v~|~ v_{(i,j,l)}\in\begin{cases}\left\{b_{(i,j,l)}\frac{\uhat_{(i,j,l)}}{|\uhat_{(i,j,l)}|}\right\}& \mbox{ if } |\uhat_{(i,j,l)}|\neq 0\\
\left\{b_{(i,j,l)}e^{\mathbf{i}\theta}\mid \theta\in[0,2\pi] \mbox{ and $\mathbf{i}^2=-1$}\right\}& \mbox{ else}\end{cases}\right\}
\label{eq:PM_theory}
\end{equation}
where $\mathcal{F}^{-1}$ is the discrete inverse Fourier transform.
Note that $P_{\M}u$ is in general a {\em set};  any element from this set will do.
This formula, since it involves division by a possibly small number, is not recommended numerically \cite[Corollary 4.3]{luke2002optical}.  Since our primary interest here is the convergence theory under the assumption of {\em exact} arithmetic, we will ignore this important detail and just caution readers against using this formula in practice.

The projector $P_{\LF}$ is given by
\begin{equation}\label{eq:P_LM}
     P_{\LF}u \equiv \mathcal{F}^{-1}(v),
     \quad\mbox{ where }\quad
    v_{(i,j,l)} =\begin{cases}
    \uhat_{(i,j,l)}&\text{ if $(\xhat_i,\yhat_j,\zhat_l)\in B_{\overline r}$}\\
    0 & \text{otherwise.}
    \end{cases}
\end{equation}
The projector $P_{\SUPP}$ is similar, and can be computed simply by applying the binary mask $\Omega$, i.e.
the projection $u^{\SUPP}\in P_\SUPP u$ is given by
\begin{equation}\label{eq:P_SUPP}
     u^{\SUPP}_{(i,j,l)} =\begin{cases}
    u_{(i,j,l)}&\text{if }\Omega_{(i,j,l)}\neq 0,\\
    0 & \text{otherwise.}
    \end{cases}
\end{equation}
Both $P_{\LF}$ and $P_{\SUPP}$ are single-valued mappings since $\LF$ and $\SUPP$ are convex.

Now for the sparse-real constraint set $\SR$.  We present the formulation derived in \cite[Proposition 3.6]{bauschke2014restricted}. Define
$ \mathcal J:=2^{\left\{1,2,\dots,2N\right\}}$ and $\mathcal J_s:= \{ J\in \mathcal J ~|~ J \mbox { has} $ $ s\mbox{ elements}\}$.
The sparse set
\begin{equation}\label{eq:sparse.set}
\sparse:=\left\{x\in \mathbb C^N ~\middle |~\|x\|_0\le s\right\}
\end{equation} can be written as the union of all subspaces indexed by $J\in \mathcal J_s$ \cite[Equation (27d)]{bauschke2014restricted},
\begin{equation}\label{sparsedecomp}
 \sparse=\bigcup_{J\in \mathcal J_s} A_J,
\end{equation}
where $A_J:=\textup{span}\left\{e_i~\middle|~i\in J\right\}$ and $e_i$ is the $i-$th standard unit vector in $\mathbb{C}^N\approx (\R^2)^N$ where
 $i\in \{1, 2, \dots, 2N\}$ with the odd indexes being the real parts and the even indexes being the imaginary parts of the points.
For $x\in\mathbb{C}^N$ we define the set of $s$ largest coordinates in absolute value
\begin{equation}\label{e:Cs}
C_s(x):=\left\{J\in \mathcal J_s ~\middle|~ \min_{i\in J}|x_i|\geq \max_{i\notin J} |x_i|\right\}.
\end{equation}
The projection of a point $x\in \mathbb{C}^N$ onto the sparsity constraint $\sparse$ is then given by
\begin{equation}\label{eq:P_A}
P_{\sparse} x\equiv \bigcup_{J\in C_s(x)} P_{A_J}x.
\end{equation}
Next, denote the real subspace of $\mathbb{C}^N$ by $\mathscr{R}^N$.  Clearly $(P_{\mathscr{R}^N}u)_j = (\text{Re}(u_j), 0)$ for $j=1,2,\dots, N$.
Now, putting this together with the real-valued constraint, it can easily be shown that the {\em projector} onto the set $\SR$ is given by
\begin{equation}\label{e:P_SR}
    P_{\SR}u= P_{\sparse}P_{\mathscr{R}^N}u.
\end{equation}
Indeed, we have
\begin{equation*}
\begin{array}{rl}
     P_\SR u=&
\argmin_
{v\in \mathbb C^N}\left\{\|v-u\|^2~|~\|v\|_0 \leq s\,,\,\text{Im}(v)=0 \right\} \\
[5pt]= & \argmin_{v\in \mathbb C^N}\left\{\|\text{Re}(v)-\text{Re}(u)\|^2 +
\|\text{Im}(v)\|^2~|~\|v\|_0 \leq s\,\right\}  \\
[5pt]
= & \argmin_{v\in \mathcal{A}_s}\left\{ \|\text{Re}(v)-\text{Re}(P_{\mathscr{R}^N}(u))\|^2 +
\|\text{Im}(v) - \text{Im}(P_{\mathscr{R}^N}(u))\|^2\right\}\\[5pt]
= & \argmin_{v\in \mathcal{A}_s}\left\{ \|v-P_{\mathscr{R}^N}(u)\|^2
\right\}
\\[5pt]
=&P_{\sparse}P_{\mathscr{R}^N}(u) \,.
\end{array}
\end{equation*}

There are two things to note about the formula \eqref{e:P_SR}:
firstly, this projector is set-valued reflecting the fact that the set $\sparse$ (and hence $\SR$) is not convex;
secondly, if the order of the operations $P_{\sparse}P_{\mathscr{R}^N}(\cdot)$ is changed, the resulting operator is {\em not} the projection onto the set $\SR$.

Finally, we derive the projector onto the symmetry constraint.
Let
$\Pi_{\sym_{j}}:\mathbb C^{N_1\times N_2\times N_3}\to
\mathbb C^{N_1\times N_2\times N_3}$ for $j=1,2,3$ be linear operators defined by
\begin{equation}
\begin{array}{rl}
&\Pi_{\sym_{1}}u\equiv u^{\Tcal_1} = [u^{\Tcal_1}_{(i,j,l)}]_{(i,j,l)\in\mathbb I}\mbox{ where }u^{\Tcal_1}_{(i, j,l)} = u_{(N_1-i+1,j,l)}\,,\\[5pt]
&\Pi_{\sym_{2}}u\equiv u^{\Tcal_2}=[u^{\Tcal_2}_{(i,j,l)}]_{(i,j,l)\in \mathbb{I}}\text{ with } u^{\Tcal_2}_{(i,j,l)}=u_{(i,N_2-j+1,l)}\,,\\[5pt]
&\Pi_{\sym_{3}}u\equiv u^{\Tcal_3}=[u^{\Tcal_3}_{(i,j,l)}]_{(i,j,l)\in \mathbb{I}}\text{ with } u^{\Tcal_3}_{(i,j,l)}=u_{(i,j,N_3-l+1)}, \,\,(i,j,l)\in \mathbb I.
\end{array}
\end{equation}
These operators are simply transposition operators in the respective coordinate axes of the tensors.

Seven lemmas characterize the sets in this application and their regularity.  The first two of these are stated here.  The remaining lemmas are somewhat more technical
and not necessary for understanding the main results.  The proofs of all of the lemmas
are given in  the appendix. 
\begin{lemma}\label{l:SYMM}
The symmetry set $\SYM$ satisfies
\begin{equation}\label{e:symm as intersection}
\SYM = \SYM^3_-\cap\SYM^2_-\cap\SYM^1_+,
\end{equation}
where
\begin{equation}\label{e:symm subsets}
  \begin{array}{rl}
       &\SYM^1_+ \equiv \{u\in\mathbb{C}^N \mid u_{(i,j,l)}=u_{(N_1-i+1,j,l)}\},\\
       &\SYM^2_- \equiv \{u\in\mathbb{C}^N \mid u_{(i,j,l)}=-u_{(i,N_2-j+1,l)}\},\\
       &\SYM^3_- \equiv \{u\in\mathbb{C}^N \mid u_{(i,j,l)}=-u_{(i,j,N_3-l+1)}\}.
  \end{array}
\end{equation}
Moreover,
\begin{subequations}
\begin{equation}\label{e:P_SYM}
P_{\SYM} = P_{\SYM^3_-}\circ P_{\SYM^2_-}\circ P_{\SYM^1_+},
\end{equation}
where
\begin{eqnarray}\label{e:P_symm1}
P_{\SYM^1_+} &=&\frac{\Id+\Pi_{\sym_{1}}}{2}\\
\label{e:P_symm2}
P_{\SYM^2_-} &=&\frac{\Id-\Pi_{\sym_{2}}}{2}\\
\label{e:P_symm3}
P_{\SYM^2_-} &=&\frac{\Id-\Pi_{\sym_{3}}}{2}.
\end{eqnarray}
\end{subequations}
\end{lemma}

\subsection{Convergence and Fixed Point Characterizations}\label{sec:convergence.guarantee}

In this section, we establish local linear convergence of all of the methods in our comparison, namely Algorithm \ref{alg:gen fpa}
when the fixed point mapping $T$ is either cyclic projections \eqref{eq:1st.CP}, cyclic relaxed Douglas-Rachford \eqref{eq:CDRl_orbit3D},
or relaxed Douglas-Rachford on the product space \eqref{eq:DRl product} applied to the 3D orbital tomography problem \eqref{eq:multiset.model.orbit}.
The convergence proof for the cyclic relaxed Douglas-Rachford algorithm for general, nonintersecting  sets is provided in \cite{DinhJansenLuke2024CDRl};  the convergence
of cyclic projections and relaxed Douglas-Rachford for general nonconvex inconsistent feasibility was obtained, respectively, in \cite{russell2018quantitative}
and \cite{luke2020convergence}.
The proofs of the specializations of these general results to the 3D orbital tomography problem considered here are given in the appendix.
We finish this section with a brief statement of the characterizations of the fixed point sets.

\subsubsection{Convergence with Rates}
The following lemma establishes that the sets above are prox-regular.  

\begin{lemma}\label{lem:prox.regular.sets}
$~$
\begin{enumerate}[(i)]
    \item\label{lem:prox.regular.sets i} 
    The set $\M$ defined by \eqref{eq:phase_reformulated} is closed and prox-regular.
    \item\label{lem:prox.regular.sets ii} The sets $\SYM$, $\SUPP$, and $\LF$ defined respectively by \eqref{eq:symmetry_set},  \eqref{eq:support_constraint}, and \eqref{eq:LM} are convex.
    \item\label{lem:prox.regular.sets iii} For each $u\in \SR\backslash\{0\}$ with $\SR$  defined by \eqref{eq:sparse_set},
    there exist a neighborhood $\check U\subset \mathbb C^N$ of $u$ and a set $\check \Lambda\subset \mathbb C^N$ such that
    $P_{\SR}$ is pointwise $\alpha$-fne with $\alpha=1/2$ and the reflector $R_\SR$ is pointwise nonexpansive at each
    $\check y'\in\check \Lambda$ (violation $0$) on $\check U$.
\end{enumerate}
\end{lemma}

The lemmas in the appendix establish one of two central assumptions for quantitative convergence of the
cyclic relaxed Douglas-Rachford algorithm for the orbital tomography reconstruction problem \eqref{eq:multiset.model.orbit}, namely the
a$\alpha$-fne property of the fixed point mapping corresponding to the algorithms under consideration here.
For this we will need additional technical assumptions.
\begin{assumption}\label{a:1}  For ease of notation, let $A_1\equiv \SYM$, $A_2\equiv \SR$, $A_3\equiv\SUPP$, $A_4\equiv \LF$ and
$A_5\equiv \M$. Consider the following conditions with $A_6\equiv A_1$, $U_6=U_1=U_3=U_4=\Lambda_6=\Lambda_1=\Lambda_3=\Lambda_4=\mathbb C^N$:
\begin{enumerate}[(a)]
\item\label{a:1,b}  On the neighborhood $U_2$ of $x_2\in A_2$, the projector $P_{A_2}$ is pointwise $\alpha$-fne with $\alpha=1/2$ at
each point in $\Lambda_2:=P^{-1}_{A_2} (A_2\cap U_2)\cap U_2$, and the set
$A_5$ is $\epsilon$-super-regular at $x_5\in A_5$ relative to the set $\Lambda_5 :=P^{-1}_{A_5} (A_5\cap U_5)\cap U_5$
with constant $\epsilon_{U_5}$ on the neighborhood $U_5$ of $x_5$.

\item\label{a:1,c} The neighborhood $U_2$ and the set
 $\Lambda_2$ satisfy
$R_{A_{3}}U_{3}\subseteq U_{2}$ and $R_{A_{3}}\Lambda_{3}\subseteq \Lambda_{2}$.
Similarly, the neighborhood $U_5$ and the set
 $\Lambda_5$ satisfy
$R_{A_{6}}U_{6}\subseteq U_{5}$ and $R_{A_{6}}\Lambda_{6}\subseteq \Lambda_{5}$.
\item\label{a:1,d}  $T_{2,3}\Lambda_{3}\subseteq \Lambda_{2}$,
$T_{2,3}U_{3}\subseteq U_{2}$, $T_{5,6}\Lambda_{6}\subseteq \Lambda_{5}$, and
$T_{5,6}U_{6}\subseteq U_{5}$, for $T_{j,j+1}\equiv T^{\text{DR}\lambda}_{A_j, A_{j+1}}$ defined by \eqref{e:T_AB}.
\item\label{a:1,d'}$\epsilon_{U_5}\in [0,4\sqrt{2}/7-5/7)$.
\item\label{a:1,e}
$P_{A_{3}}U_{3}\subseteq U_{2}$ and $P_{A_{3}}\Lambda_{3}\subseteq \Lambda_{2}$.
\item\label{a:1,f} $\epsilon_{U_5}\in [0,2\sqrt{3}/3-1)$.
\end{enumerate}
\end{assumption}

The following general linear convergence result relies on three properties of the fixed point mapping $T$, generally described as
{\em existence}, {\em stability}, and {\em almost quasi-contractivity}.
\begin{assumption}[regularity]\label{ass:regularity}
  Let $T:\Lambda\rightrightarrows \Lambda$ for $\Lambda\subseteq\Ebb$.  Let $U\subset\Ebb$
  with $U\cap\Lambda\neq\emptyset$.
  The following assumptions hold.
  \begin{enumerate}[(a)]
  \item \label{t:msr convergence c}  (Existence) There is at least one
  $\xbar \in \Fix T\cap \Lambda$.
  \item\label{t:msr convergence b} (Stability) There exists a $\kappa>0$ such that
   \begin{equation}\label{e:msr}
    \dist(x,\Fix T\cap \Lambda)\leq \kappa\dist(x, Tx)\quad\forall x\in U\cap \Lambda.
   \end{equation}
  \item\label{t:msr convergence a} (Almost quasi-contractivity) $T$ is pointwise a$\alpha$-fne at all $y\in \Fix T\cap\Lambda$,
  that is, $T$ satisfies
  \begin{eqnarray}\label{eq:paafne}
    &&\exists \epsilon\in [0,1), \alpha\in (0,1):\quad \forall x\in U, \forall y\in \Fix T\cap \Lambda\\
    &&\quad\norm{T x - y}^2\leq
    (1+\epsilon)\|x-y\|^2 -
    \tfrac{1-\alpha}{\alpha}\norm{x - T x}^2.
    \nonumber
  \end{eqnarray}
\end{enumerate}
\end{assumption}

\begin{assumption}\label{a:msr convergence}
\begin{equation}\label{eq:theta linear}
 0< \gamma\equiv \sqrt{1+\epsilon-\frac{1-\alpha}{\alpha\kappa^2}}<1\quad\iff\quad
\sqrt{\tfrac{1-\alpha}{\alpha(1+\epsilon)}}\leq  \kappa\leq \sqrt{\tfrac{1-\alpha}{\alpha\epsilon}}.
\end{equation}
\end{assumption}
The lower bound on $\kappa$ in \eqref{eq:theta linear} is easily satisfied; it is the upper bound that could be difficult.

The next result is a restatement of convergence rates \cite[Theorem 2.6]{HerLukStu23b} for the setting of multi-valued mappings
in Euclidean spaces. The statement of \cite[Theorem 2.6]{HerLukStu23b} concerns random selections from collections of
single-valued mappings, but the extension to multi-valued mappings presents no difficulty here since
we are in the deterministic case,
and logic of the proof, specialized to a single mapping, works here as well.
\begin{proposition}[Theorem 2.6 \cite{HerLukStu23b}]\label{t:metric.subreg.convergence} Let $T:\Lambda\rightrightarrows \Lambda$ for
$\Lambda\subseteq\Ebb$ satisfy Assumption \ref{ass:regularity}.
For any $x^{(k)}\in U\cap\Lambda$ close enough to $\Fix T\cap \Lambda$, the iterates $x^{(k+1)}\in Tx^{(k)}$ satisfy
\begin{equation}\label{e:iterate bound}
\dist(x^{(k+1)}, \Fix T\cap \Lambda)
\leq \gamma\dist\paren{x^{(k)}, \Fix T\cap \Lambda}
\end{equation}%
where $\gamma=\sqrt{(1+\epsilon)-\frac{1-\alpha}{\alpha\kappa^2}}$.
If in addition Assumption \ref{a:msr convergence} is satisfied with parameter values
$\overline\alpha$ and $\overline\epsilon$, then for all $x^{(0)}\in U$ close enough to
$\Fix T\cap \Lambda$, $x^{(k)}\to \overline x\in \Fix T\cap \Lambda$
R-linearly with  rate constant $\overline\gamma=\sqrt{(1+\overline\epsilon)-\frac{1-\overline\alpha}{\overline\alpha\kappa^2}}<1$.
\end{proposition}
We apply this to the present problem.  The proofs of the following theorems are in the appendix.

\begin{theorem}\label{t:CDRl Orbital convergence}
For the cyclic relaxed Douglas-Rachford mapping $T^{\text{CDR}\lambda}$ defined by \eqref{eq:CDRl_orbit3D}
let Assumption \ref{a:1} \eqref{a:1,b}-\eqref{a:1,d'}, and Assumption \ref{ass:regularity}\eqref{t:msr convergence c}-\eqref{t:msr convergence b} hold.
Additionally, let Assumption \ref{a:msr convergence}
hold with parameter values $\overline\alpha=5/6$ and $\overline\epsilon$ given by \eqref{e:CDRl alpha-epsilon}, where these
are the constants characterizing the regularity of the a$\alpha$-fne mapping $T^{\text{CDR}\lambda}$.
Define the sequence $(x^{(k)})$ by $x^{(k+1)}\in T^{\text{CDR}\lambda} x^{(k)}$.
Then for all $x^{(0)}$ close enough to
$\Fix T^{\text{CDR}\lambda}\cap \Lambda$, $x^{(k)}\to \overline x\in \Fix T^{\text{CDR}\lambda}\cap \Lambda$
R-linearly with  rate constant $\overline\gamma=\sqrt{(1+\overline\epsilon)-\frac{1-\overline\alpha}{\overline\alpha\kappa^2}}<1$.
\end{theorem}

\begin{theorem}\label{t:CP Orbital convergence}
For the cyclic projection mapping $T^{\text{CP}}$ defined by \eqref{eq:first.CP},
let Assumption \ref{a:1} \eqref{a:1,b}, \eqref{a:1,e}, \eqref{a:1,f} and Assumption \ref{ass:regularity}\eqref{t:msr convergence c}-\eqref{t:msr convergence b} hold.
Additionally, let Assumption \ref{a:msr convergence}
hold with parameter values $\overline\alpha=5/6$ and $\overline\epsilon$ given by \eqref{e:CP alpha-epsilon}, where these
are the constants characterizing the regularity of the a$\alpha$-fne mapping $T^{\text{CP}}$.
Define the sequence $(x^{(k)})$ by $x^{(k+1)}\in T^{\text{CP}} x^{(k)}$.
Then for all $x^{(0)}$ close enough to
$\Fix T^{\text{CP}}\cap \Lambda$, $x^{(k)}\to \overline x\in \Fix T^{\text{CP}}\cap \Lambda$
R-linearly with rate constant $\overline\gamma=\sqrt{(1+\overline\epsilon)-\frac{1-\overline\alpha}{\overline\alpha\kappa^2}}<1$.
\end{theorem}

A general convergence proof for the relaxed Douglas-Rachford algorithm on the product space is provided in \cite{luke2020convergence}. A simpler
proof is provided in the appendix.
\begin{theorem}\label{t:DRl Orbital convergence}
For $T^{\textbf{DR}\lambda}$ defined by \eqref{eq:DRl product},
let Assumption \ref{a:1} \eqref{a:1,b}, \eqref{a:1,d'}, and Assumption \ref{ass:regularity}\eqref{t:msr convergence c}-\eqref{t:msr convergence b} hold for $\overline x$ replaced by $\overline{\mathbf x}$ and $\mathbb E$ replaced by $\mathbb E^m$.
Additionally, let Assumption \ref{a:msr convergence}
hold with parameter values $\overline\alpha=1/2$ and $\overline\epsilon$ given by 
\begin{equation}
    \overline \epsilon:=\frac{1}{2}\left[\left(\lambda\sqrt{1+\tilde \epsilon_5}+1-\lambda\right)^2-1\right],
\end{equation}
where
$\tilde{\epsilon}_5= 8\epsilon_{U_5}(1 + \epsilon_{U_5})/(1 -\epsilon_{U_5})^2$, and $\overline \alpha, \overline \epsilon$
are the constants characterizing the regularity of the a$\alpha$-fne mapping $T^{\textbf{DR}\lambda}$.
Define the sequence $(\mathbf x^{(k)})$ by $\mathbf x^{(k+1)}\in T^{\textbf{DR}\lambda} \mathbf x^{(k)}$.
Then for all $\mathbf x^{(0)}$ close enough to
$\Fix T^{\textbf{DR}\lambda}\cap \Lambda$, the sequence $\mathbf x^{(k)}\to \overline {\mathbf x}\in \Fix T^{\textbf{DR}\lambda}\cap \Lambda$
R-linearly with  rate constant $\overline\gamma=\sqrt{(1+\overline\epsilon)-\frac{1-\overline\alpha}{\overline\alpha\kappa^2}}<1$.
\end{theorem}

%
%
\subsubsection{Characterization of Fixed points}\label{sec:characterizations}
For reference, we provide the characterizations of the fixed point sets of the three algorithms.  These
have been established in earlier work.  We simply state the results here.

Based on our work in \cite{DinhJansenLuke2024CDRl}, the characterization of the fixed points of the cyclic relaxed Douglas-Rachford operator for the reconstruction of the 3D orbital problem is given by \cite[Corollary 3.1 and Remark 3.3]{DinhJansenLuke2024CDRl}
\begin{eqnarray}
\emptyset=\SYM\cap \SR\cap \SUPP\cap\LF\cap \M&
        \overset{{\lambda\in[0,1]}}{\subseteq}&
        \Fix T^{\text{CDR}\lambda}\nonumber\\
        &\overset{{\lambda\in[0,\tfrac12]}}{\subseteq}&
        \conv \left(\SYM\cup \SR\cup \SUPP\cup\LF\cup \M\right),
    \end{eqnarray}
and \cite[Eq(30)]{DinhJansenLuke2024CDRl}
\begin{equation}\label{eq:shadow.sunset.CP.noise.thm}
    P_{\SYM}(\Fix T^{\text{CDR}\lambda}\cap U)\subseteq \Fix_{\!\!\epsilon}\, T^{\text{CP}}\equiv\{y\in\mathbb C^N\,:\, y\in T^{\text{CP}}y +\varepsilon \mathbb B\},
\end{equation}
where $U$ is a subset of $\mathbb C^N$ intersecting $\Fix T^{\text{CDR}\lambda}$ to which the iterates are automatically confined, and $\mathbb B$ is the closed unit ball in $\mathbb C^N$.

The fixed points of the relaxed Douglas-Rachford operator on the product space 
were characterized in \cite{luke2020convergence}.  This is specialized to the 
present application:
\begin{equation}\label{eq:Fix DRl product space}
    \begin{aligned}
     & \Fix T^{\textbf{DR}\lambda}\cap U \\
     &\qquad \subseteq \left\{ f-\frac{\lambda}{1-\lambda} (f-e)\,:\, f\in P_{\Omega}(f-\frac{\lambda}{1-\lambda} (f-e))\,,\,e\in P_D(f) \right\}\cap U,
    \end{aligned}
\end{equation}
    where $\Omega = \SYM\times \SR\times \SUPP\times\LF\times \M$, $D$ is the diagonal set in $(\mathbb C^N)^5$, $U$ is an open set in $(\mathbb C^N)^5$, and
\begin{equation}\label{eq:monotoncity}
    P_{\Omega}(\Fix T^{\textbf{DR}\lambda_2})\subseteq P_{\Omega}(\Fix T^{\textbf{DR}\lambda_1})
    \quad \mbox{whenever }\lambda_1\le \lambda_2.
\end{equation}
The main message of these characterizations is that the shadows of the fixed point sets parameterized by $\lambda$,
$P_{\Omega}(\Fix T^{\textbf{DR}\lambda})$,  are
monotonically decreasing as $\lambda$ increases.  The fixed points of $T^{\text{CDR}\lambda}$ are in the convex hull of the
union of the constraint sets, and remain ``near'' the fixed points of the cyclic projections mapping.  The significance of this is demonstrated
in the numerical demonstrations below.

\section{Numerical experiments}
\label{sec:experiments}
This section presents numerical results for the algorithms applied to a simulated and laboratory  photoemission orbital tomography experiments.  The
data and scripts for running the simulated experiment are available at \cite{Dinh_2024}.  The experimental data is available from \cite{Bennecke25_data}.

\subsection{Simulated Data}
We used the simulated dataset with 13 spheres to evaluate the reliability of the output with this amount of input.
The dataset contains 9.9\% non-zero voxels. The remaining 90.1\% of the voxels need to be interpolated, see Figure\ref{fig:data.best.reconstruction}(a).

We collect the behavior of the algorithms from 100 random starting points. The sparsity parameter is $s=2400$.
The support areas in the physical domain and the Fourier domain are set the same as in \cite{Dinh_2024}. 

In Figure \ref{fig:data.best.reconstruction}(b)-(e), we compare the best approximated fixed point to the exact solution in both the Fourier domain and the physical domain using the cyclic relaxed Douglas-Rachford with $\lambda=0.7$.
As expected, they appear nearly identical, except that the charge is flipped.
\begin{figure}
\subfigure[]{
\includegraphics[scale=0.5]{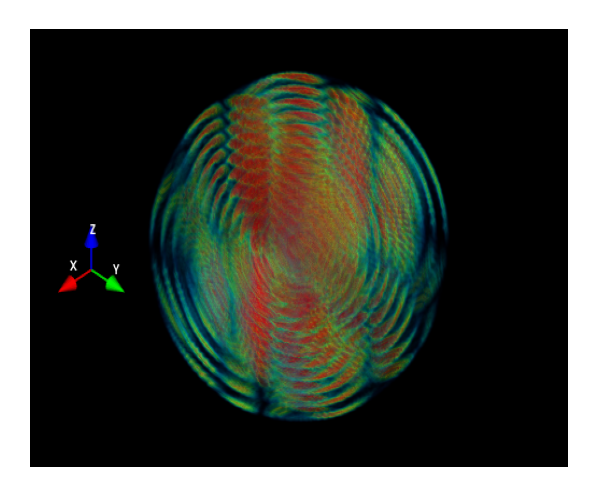}}
\subfigure[]{\includegraphics[scale=0.5]{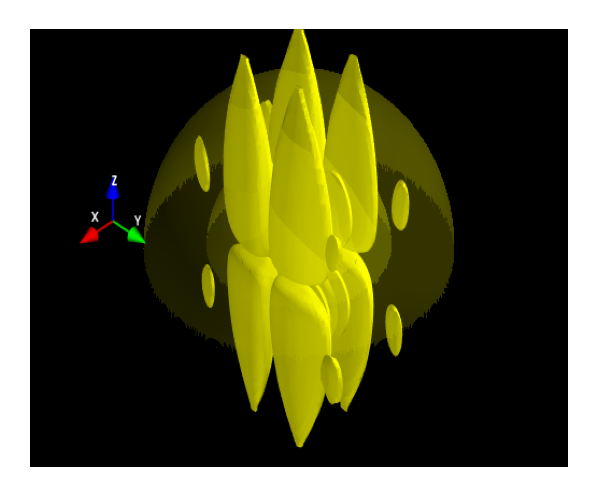}}
\subfigure[]{\includegraphics[scale=0.5]{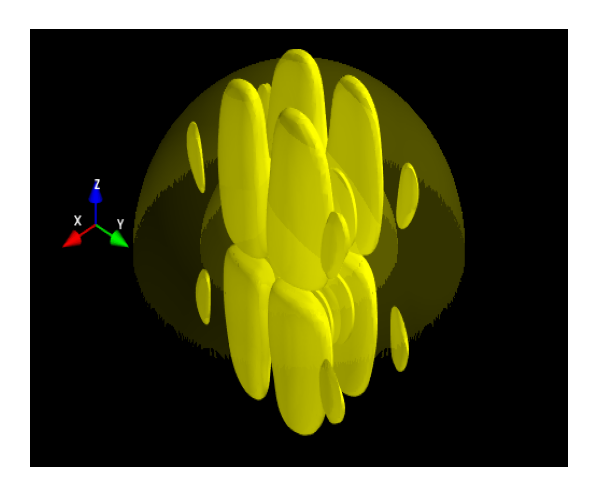}}
\subfigure[]{\includegraphics[scale=.85]{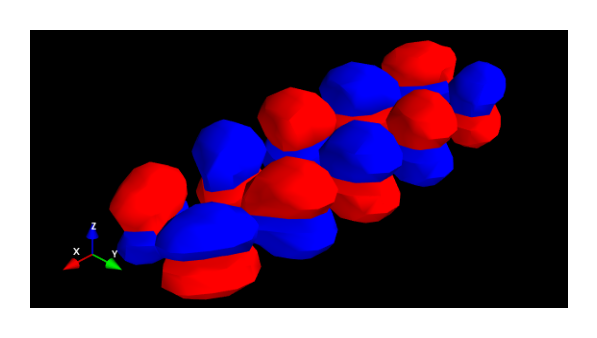}}
\subfigure[]{\includegraphics[scale=.85]{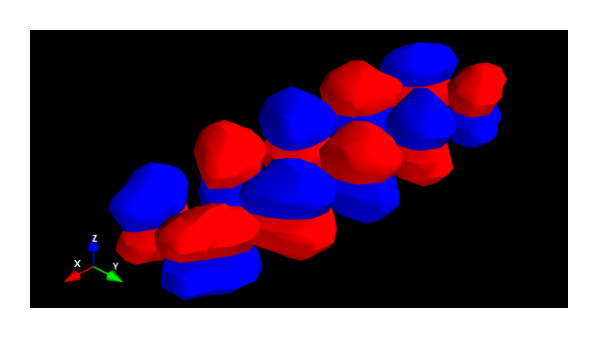}}
\caption{a) Data obtained from ARPES; b) Fourier amplitude of the exact solution; c) Fourier amplitude of reconstruction at smallest gap of the CDR$\lambda$ with $\lambda=0.7$;
d) The exact solution; e)
Reconstruction at smallest gap of the CDR$\lambda$ with $\lambda=0.7$.}
\label{fig:data.best.reconstruction}
\end{figure}
In Figure \ref{fig:changes.10times}(a), we present five examples of convergence plots for the cyclic projection and the cyclic relaxed Douglas-Rachford methods
($\lambda=0.7$) starting from the same initializations.
Once a local region of convergence has been found, the algorithm converges linearly to a tolerance $\text{tol}=10^{-8}$.  The local rate of convergence for each algorithm appears to be independent
of which local minimum the algorithm is converging to, indicating that the local minima have the same regularity/geometry.  The theory behind these methods, whether cyclic projections or Douglas-Rachford, does not state how long one has to wait before a local region of convergence is found (i.e. there is no global theory).
The slopes of the convergence plots for cyclic projections are always steeper than those of the cyclic Douglas-Rachford, indicating that cyclic projections has a better rate of linear convergence.
This confirms the theoretical and numerical studies of the two-set instance of these algorithms showing that the Douglas-Rachford algorithm has a poorer constant of local linear convergence than cyclic projections \cite{HesseLukeNeumann14, BauschkeBelloCruzNghiaPhanWang14}.
Over 100 trials, the cyclic projection method reaches the tolerance in an average of 169 iterations, while the cyclic relaxed Douglas-Rachford method requires more time, averaging 336 iterations.
If we reduce $\lambda$ to $0.3$, the average number of iterations is reduced to $280$.
This can be explained based on the operator: as $\lambda \to 0$, the relaxation term is smaller, and the cyclic relaxed Douglas-Rachford operator approaches the cyclic projection operator.
In Figure \ref{fig:changes.10times}(b), we illustrate the relationship between the \emph{gap} \eqref{e:gap} and the \emph{error} \eqref{eq:error.with.truth} at final iterates for the two algorithms, where the error represents the difference between the approximate fixed point and the exact solution (see \cite{Dinh_2024}).
We observe that: 1) a smaller gap indicates better reconstruction, and 2) the approximate fixed points of the cyclic relaxed Douglas-Rachford have gaps and errors that are roughly the same as those of the fixed points of the cyclic projection.
These observations further support the statement that the gap is a reliable indicator of quality, and that the shadow of the fixed points of the cyclic relaxed Douglas-Rachford on the set $\SYM$ is close to the fixed points of the cyclic projection, as stated in Section \ref{sec:characterizations}.

\begin{figure}
    \centering
    \subfigure[]{
\includegraphics[scale=0.42]
{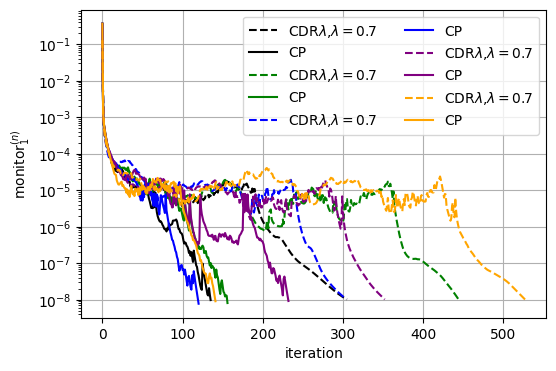}}
\subfigure[]{
\includegraphics[scale=0.42]{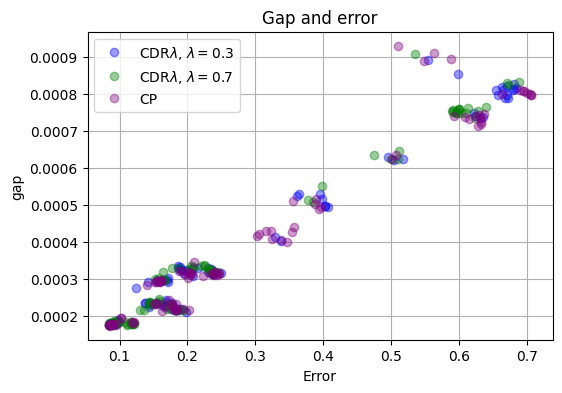}}
    \caption{(a) Five ``typical'' convergence plots of the cyclic relaxed Douglash-Rachford and cyclic projection algorithms. (b) The relationship between gap and error at (approximate) fixed points of the two algorithms.}
    \label{fig:changes.10times}
\end{figure}
We now compare the efficiency of the two methods.
\begin{figure}
\centering
\subfigure[]{
\includegraphics[scale=0.42]{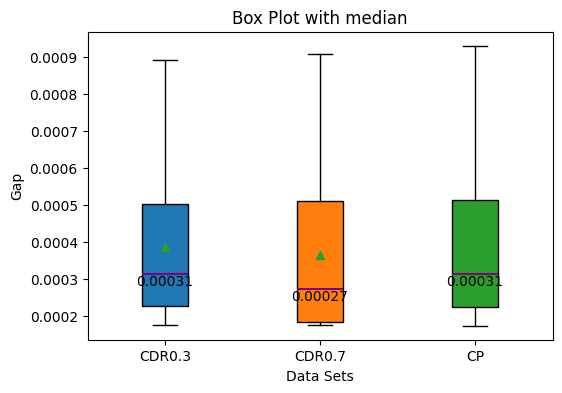}}
\subfigure[]{
\includegraphics[scale=0.42]{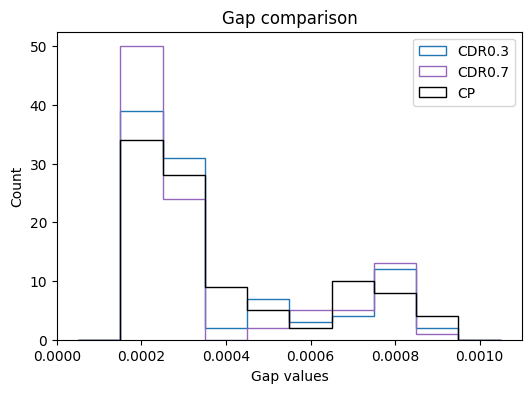}}
\caption{CP and CDR$\lambda$, $\lambda=0.7, 0.3$ with 100 trials for each method. (a) Boxplots showing gap variation. (b) Histogram of the gap.}
\label{fig:histogram}
\end{figure}
Figure \ref{fig:histogram}(a) shows the median, mean, variance, maximum and minimum of the gaps of the final iterates of the cyclic projections and cyclic relaxed Douglas-Rachford algorithms ($\lambda=0.3$ and $\lambda=0.7$).  The minimum achieved gap for all three numerical experiments is the same, but the mean, variance and maximum gaps are slightly different, indicating that CDR$\lambda$ with $\lambda=0.7$ achieves (slightly) better results on average.  This is
also shown in the histogram Figure \ref{fig:histogram}(b) of gap values obtained by the three. We divide the approximate fixed points into 8 clusters and see that
the cyclic relaxed Douglas-Rachford achieves a higher frequency of small gap values compared to the cyclic projection for each value of $\lambda$.
In the cluster with the smallest gap values, while the cyclic projection algorithm finds this cluster only 34\% of the time, the cyclic relaxed Douglas-Rachford finds this cluster for both values of $\lambda$, reaching up to 50\% at $\lambda=0.7$.
These results demonstrate that the cyclic relaxed Douglas-Rachford algorithm has a greater ability to achieve a higher success rate than the cyclic projection, which is considered state-of-the-art in \cite{Dinh_2024}.
This capability depends on the relaxation parameter $\lambda$, which can be fine-tuned.

We have already seen that adjusting the value of the parameter $\lambda \in [0,1]$ increases the likelihood of obtaining good reconstructions for cyclic relaxed Douglas-Rachford.
However, the ability to avoid poor local minima remains unremarkable. In the following numerical experiments, we compare both CP and CDR$\lambda$ to the relaxed Douglas-Rachford
on the product space \eqref{eq:DRl product}. 
\begin{figure}
\centering
\subfigure[]{
\includegraphics[scale=0.42]{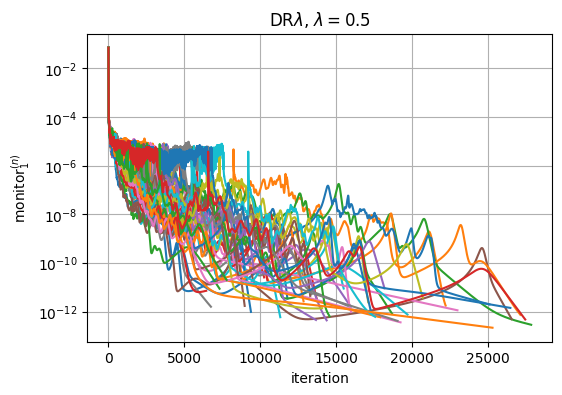}}
\subfigure[]{
\includegraphics[scale=0.42]{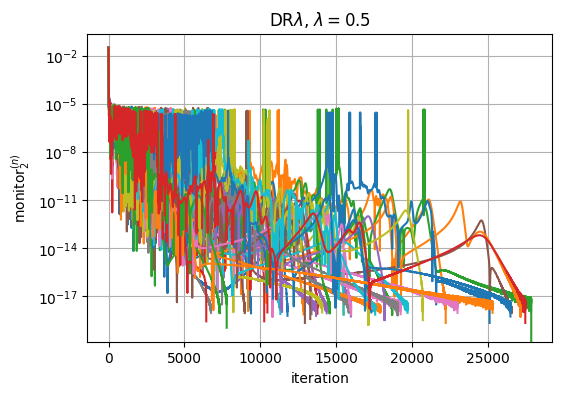}}
\subfigure[]{
\includegraphics[scale=0.42]{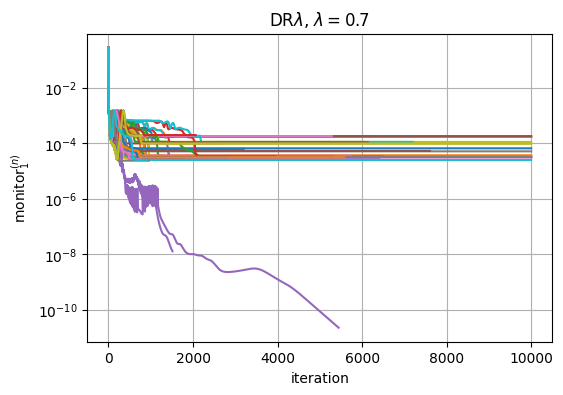}}
\subfigure[]{
\includegraphics[scale=0.42]{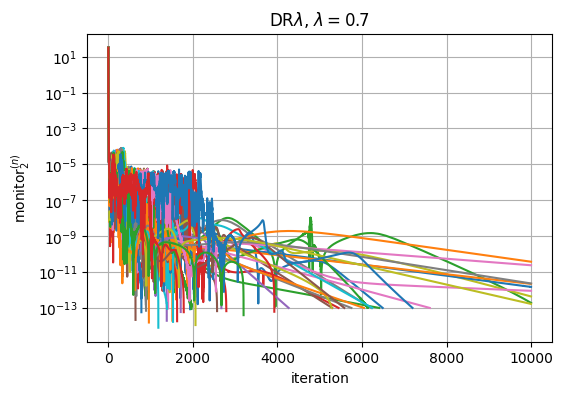}}
\caption{DR$\lambda$. Convergence plots of changes (monitor$_1$, left column) and gaps (monitor$_2$, right column) for the algorithm with two values of $\lambda$. Each $\lambda$ was run with 50 different starting points.
}
\label{fig: comparison.CP.DRl}
\end{figure}
We begin by discussing the convergence since the convergence behavior of DR$\lambda$ on the product space is qualitatively different than CP or CDR$\lambda$.

Starting from a fixed, randomly selected initial point and examining the behavior of DR$\lambda$ for different values of $\lambda \in (0,1)$,
we  observe that the convergence of this algorithm depends, as expected,  on the value of $\lambda$.
When $\lambda \leq 0.5$, the iterate difference, \emph{$\text{monitor}^{(n)}_1$}, converges linearly, as expected.
However, significantly more iterations are required to observe this convergence:  on the order of $20000$ iterations compared to less than $500$ iterations for
all of the cyclic methods.

When $0.5<\lambda <0.8$, we observe that \emph{$\text{monitor}^{(n)}_1$} lies on spheres centered at $0$, with a constant radius $r \approx \|P_\SYM u^{(n)} - P_\SYM u^{(n-1)}\|$ as $n$ approaches infinity.
Despite that, we see the convergence of the difference gap, i.e., \emph{$\text{monitor}^{(n)}_2$} in both situations. 
To confirm these behaviors, we run the algorithm with two values, $\lambda = 0.5$ and $\lambda = 0.7$. Each $\lambda$ is tested with 50 trials, using the same starting points for both values.
Figure \ref{fig: comparison.CP.DRl} confirms the initial observations more definitively.
For $\lambda = 0.5$, the changes in all the trials converge linearly, whereas with $\lambda = 0.7$, only one of them converges, see Figures \ref{fig: comparison.CP.DRl}(a) and (c).
In contrast, the gaps converge well, see \ref{fig: comparison.CP.DRl}(b) and (d).
Note that in these two experiments, we set a tolerance of $5 \times 10^{-18}$ and $n_{\max}=35000$ for $\lambda = 0.5$ and $1 \times 10^{-13}$ and $n_{\max}=10000$ for $\lambda = 0.7$, as this tolerance and maximum iteration are sufficient to observe linear convergence.

Figure \ref{fig:gap.error.CP.DRl} presents the results of the cyclic projection and relaxed Douglas-Rachford methods with the two different values of $\lambda$ for 50 trials. 
The (approximate) fixed points of CP and DR$\lambda$ on the product space with $\lambda = 0.5$ seem distributed across many clusters.
With this value of $\lambda$, the results of the two algorithms appear to be similar, but the DR$\lambda$ algorithm still excels at finding smaller gaps.
With $\lambda = 0.7$, however, DR$\lambda$ on the product space is much more robust, with most (approximate) fixed points moving to the lowest cluster,
where the gap and error are smaller.

To demonstrate the effect of the relaxation parameter on the quality of the fixed points, we conducted 50 trials for the relaxed Douglas-Rachford method initialized from the 50 different (approximate) fixed points of the algorithm with $\lambda = 0.5$ found in the previous numerical experiment.
Figure \ref{fig:gap.error.CP.DRl}(b) shows the results.
Despite different initial starting points, the algorithm tends to converge to good clusters, i.e., it returns good reconstructions.
This demonstrates the stability of the algorithm.
These findings conclude that DR$\lambda$ on the product space reliably avoids poor local minima when an appropriate value of the parameter $\lambda$ is chosen.
As characterized in Section \ref{sec:characterizations}, the set of fixed points of DR$\lambda$ on the product space
is monotone decreasing as $\lambda$ increases.   Though the characterization of the fixed points does not imply that the smaller
fixed point set {\em necessarily} corresponds to better local minima,  Figure \ref{fig:gap.error.CP.DRl} demonstrates this dramatically, showing that
the smaller fixed point set for $\lambda=0.7$ corresponds to points with smaller gaps, and errors.
\begin{figure}[hb]
\centering
\subfigure[]{
\includegraphics[scale=0.42]{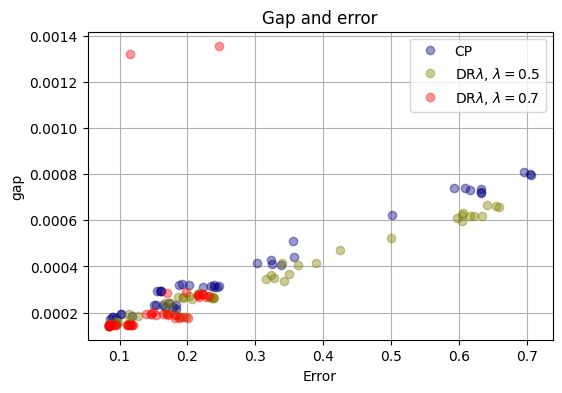}}
\subfigure[]{
\includegraphics[scale=0.42]{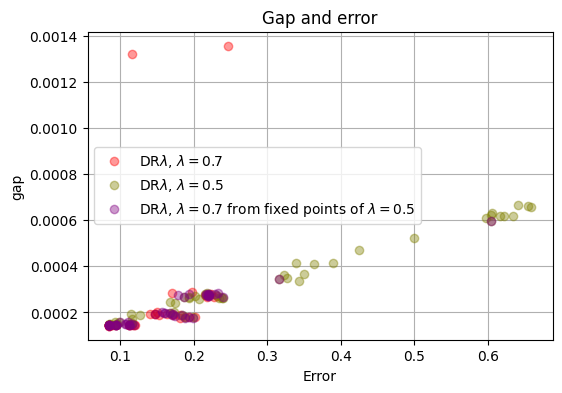}}
\caption{(a) Comparison of the efficiency between cyclic projection and relaxed Douglas-Rachford with $\lambda = 0.5$ and $\lambda = 0.7$. (b) Stabilization of approximate fixed points of relaxed Douglas-Rachford, with $\lambda = 0.7$, starting from random initialization, and the approximate fixed points of the relaxed Douglas-Rachford with $\lambda = 0.5$.
}
\label{fig:gap.error.CP.DRl}
\end{figure}

\subsection{Laboratory Data}
We perform numerical comparisons on the experimental dataset presented in \cite{Bennecke25}. All sets and parameters from \cite{Bennecke25_data} are retained here.  Since the ground truth is unknown, we compare the gap sizes for fixed points of the cyclic projection algorithm with the gaps achieved by the  product space formulation of the relaxed Douglas-Rachford method, initialized from the fixed points of CP.
\begin{figure}[h!]
\centering
\subfigure[]{
\includegraphics[scale=0.42]{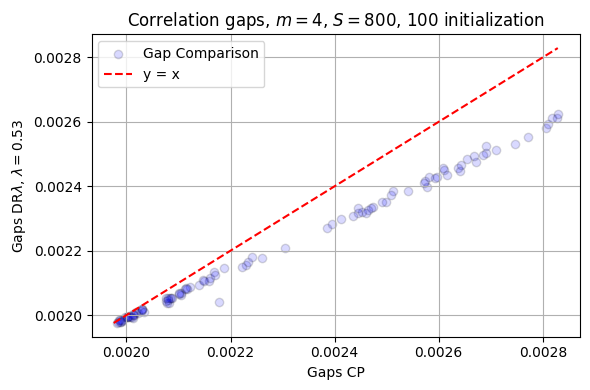}}
\subfigure[]{
\includegraphics[scale=0.42]{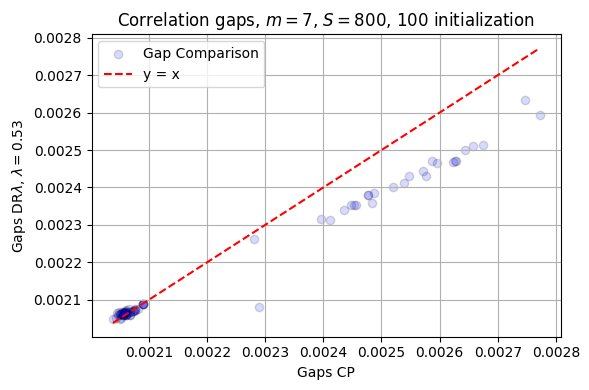}}
\subfigure[]{
\includegraphics[scale=0.42]{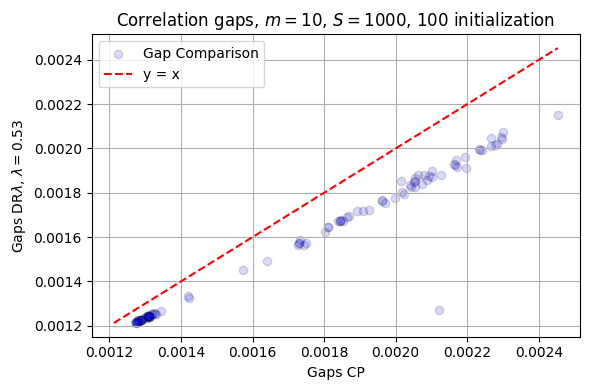}}
\caption{Comparison of the quality of fixed points between cyclic projections and product space formulation of the relaxed Douglas-Rachford algorithm with $\lambda = 0.53$.
The points show the gap sizes  achieved by DR$\lambda$ (y-axis) when initialized from the fixed points of CP (gap size shown on the x-axis).
}
\label{fig:correlation.gap.CP.DRl}
\end{figure}
\FloatBarrier
Figure \ref{fig:correlation.gap.CP.DRl} shows the correspondence of the gapsizes  for the two methods using 4, 7 and 10 data spheres, denoted as $m$.
We run the CP algorithm with 100 initializations and subsequently start the product space DR$\lambda$ algorithm from the corresponding
100 fixed points of the CP algorithm, with $\lambda=0.53$ in each case.  This parameter was found by trial and error and yielded the best results for  the method.  If a point lies on the red diagonal lines, it indicates that the DR$\lambda$ algorithm did not improve the gap over what was already achieved by the CP algorithm.  In all three plots, however, we observe that all 100 points lie below the diagonal, indicating that the relaxed Douglas-Rachford method always achieves a smaller gap than the cyclic projection. Furthermore, the DR$\lambda$ gap tends to decrease more sharply as the CP gap increases, suggesting that applying DR$\lambda$ after CP will have a more significant impact for problems with a relatively broad distribution of fixed points.  Interestingly, and quite unexpected, the experimental data appears to lead to a feasibility problem with a comparatively small range of fixed point gaps relative to the simulated data in the previous section.  In other words, the experimental data set is more regular than our simulated data set.  This apparent regularity of the experimental data can be understood as, perhaps, an averaging due to noise which eliminates regions with a large gap, understood as extreme local minima.  This makes intuitive sense in light of the characterization of the fixed points of the DR$\lambda$ algorithm given in \eqref{eq:Fix DRl product space}:  fixed points associated with large gaps should be less stable.  The difference between the structures reconstructed from the CP fixed points and those reconstructed from the subsequent DR$\lambda$ fixed points are most of the time not distinguishable physically  for this data (see Figure \ref{fig:recon comparison}(b)-(c)).  As shown in Figure \ref{fig:correlation.gap.CP.DRl}(b)-(c), however, in a small number of instances, the DR$\lambda$ algorithm can shift the CP fixed point to a significantly different structure.  This is demonstrated in \ref{fig:recon comparison}(d)-(e). 
\begin{figure}
\begin{tabular}{ccc}
\subfigure[]{
\includegraphics[scale=.7]{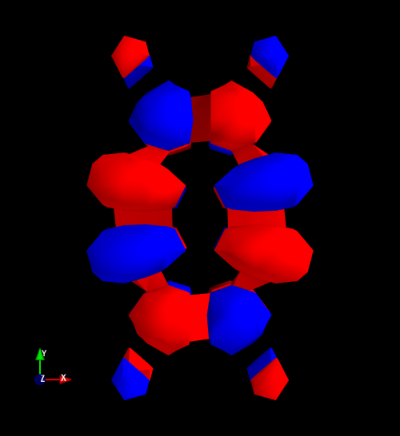}}
&
     \subfigure[]{
 \includegraphics[scale=.7]{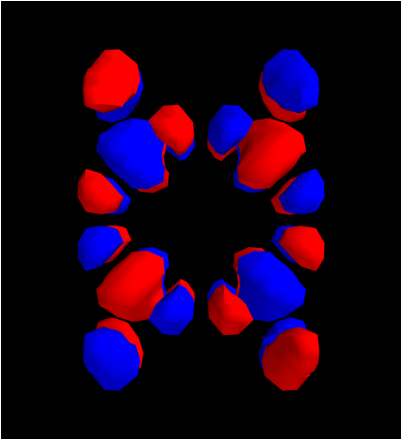}}
 &
\subfigure[]{
\includegraphics[scale=.7]{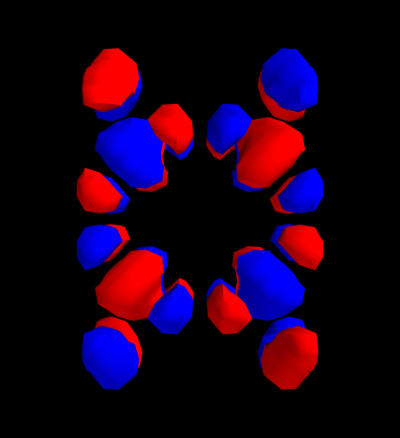}}\\
$~$
&\subfigure[]{
\includegraphics[scale=.7]{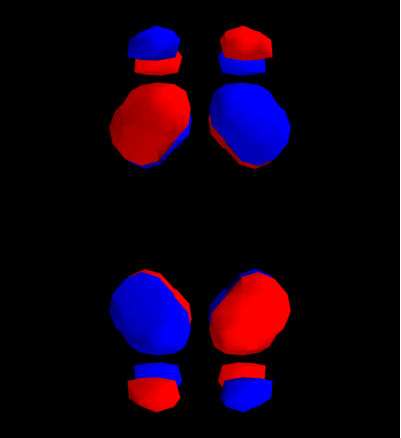}}
&
\subfigure[]{
\includegraphics[scale=.7]{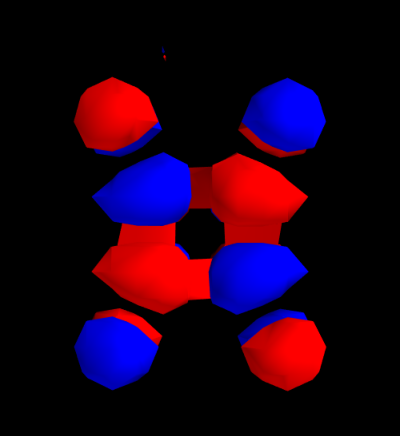}}
\end{tabular}
\caption{Examples of reconstructions using 10 data spheres. (a) A reconstruction with a small gap corresponding to a physically expected structure. (b) Reconstruction at a fixed point of CP with a large gap. (c) Reconstruction at the fixed point  of DR$\lambda$ ($\lambda=0.53$) initialized from the CP fixed point shown in (b).
(d) Reconstruction at a fixed point of CP with a large gap (gap=$0.00212$). (e) Reconstruction at the fixed point  of DR$\lambda$ (gap=$0.00127$) initialized from the CP fixed point shown in (d).
}
\label{fig:recon comparison}
\end{figure}

\section{Conclusion}
In a 2019 numerical survey projection-reflection methods, as opposed to smooth optimization techniques, were shown to be the fastest, most stable and most reliable methods for
solving nonconvex optimization problems of within a number of problem classes of a certain type \cite{luke2019optimization}.  Amongst the projection methods a
cyclic, relaxed Douglas-Rachford method \cite{luke2018relaxed} appeared to be the top performer, though the theory for this algorithm was not fully understood.
In a companion paper to the present
numerical study, the fixed points of the cyclic relaxed Douglas-Rachford algorithm, together with convergence guarantees and rates for inconsistent, nonconvex feasibility
has been fully established \cite{DinhJansenLuke2024CDRl}.  The study conducted here completes our investigation of the cyclic relaxed Douglas-Rachford algorithm, with a
detailed numerical study of its performance on a high-impact problem from experimental physics with an intricate local-minima structure that tests the limits of the leading algorithms.

The results do confirm the initial motivation
for this algorithm, but not dramatically.  The relaxed Douglas-Rachford algorithm on the product space, on the other hand, while exhibiting quite poor convergence
rates, does an excellent job of filtering out bad local minima from all cyclic algorithms.  Both the product space and cyclic implementations of the Douglas-Rachford
algorithm require significantly more iterations to find fixed points than the cyclic projections algorithm.  This is to be expected, since the iterates of the relaxed
Douglas-Rachford algorithm, for a large enough relaxation parameter, will be dispersed across a broader region of the domain than cyclic projections.
Our numerical experiments therefore lead to the following recommendation: run cyclic projections to find some fixed point, and from this fixed point run DR$\lambda$ on the
product space formulation of the problem with as large a value of $\lambda$ as is numerically stable, in order to move toward a fixed point with smaller gaps between
the constraint sets.  This advice runs counter to the current practice for phase retrieval, where the Douglas-Rachford algorithm (known in the physics community as the
Fienup Hybrid-Input-Output algorithm) is run for several iterations, and then cyclic projections (known as Gerchberg-Saxton or error reduction methods) is used to
``clean up'' the images \cite{Luke17}.


\newpage
\section*{Appendix: Proofs}\label{s:proofs}
In a classic ``chicken-and-egg'' scenario, the regularity of the projector onto a set is characterized by the regularity
of the set, which is characterized by the {\em set of normals} to the set relative to a given point, which is constructed by
the projection onto the set.  The set of {\em proximal normals} to a set $\Omega\subset\Ebb$ at $\bar{a}\in \Omega$ is defined by
		 	$$ N^P_\Omega(\bar{a}) := P^{-1}_\Omega\bar{a}-\bar{a}.$$
This definition is slightly different than the usual definition, which includes the cone generated by the vectors defined above.
When $\bar{a}\not\in \Omega$, the set of proximal normals at $\bar{a}$ is empty (by definition).
\begin{definition}\label{d:esuperreg}
A set $\Omega\subset \Ebb$ is called {\em $\epsilon$-super-regular at a distance relative
to $\Lambda\subset \Ebb$ at $\bar x$} with constant $\epsilon$ whenever
\begin{eqnarray}
    \exists ~U_\epsilon\subset\mathbb{E}~\mbox{ open}~ :~ &&\bar x\in U_\epsilon,\mbox{ and }\nonumber\\
    && \langle{v - (y'-y),y - x\rangle}\leq
    \epsilon_U\|v - (y'-y)\|\|y - x\|\label{eq:ele.subreg}\\
    &&\forall y'\in U_\epsilon\cap \Lambda, \quad \forall y\in P_\Omega(y'),\nonumber\\
    && \forall (x,v)\in\left\{(x,v)\in \gph \pncone {\Omega}~\mid~x+v\in U_\epsilon,~x\in P_\Omega(x+v)\right\}.
    \nonumber
\end{eqnarray}
The set $\Omega$ is called {\em super-regular at a distance relative
to $\Lambda$ at $\bar x$} if it is $\epsilon$-super-regular at a distance relative
to $\Lambda$ at $\bar x$ for all $\epsilon>0$.
\end{definition}
Most of the sets in the orbital tomography application are convex (Lemma \ref{lem:prox.regular.sets}),
so classical results show that the projectors are
everywhere superregular at a distance with $\epsilon=0$.  It is shown in Lemma \ref{lem:pane.refectors} that
the projectors onto these sets are $\alpha$-fne everywhere with $\alpha=1/2$, agreeing with classical results.
The data set $\M$ is not convex, but it is easily seen in the proof of Lemma \ref{lem:prox.regular.sets} to be {\em prox-regular}.
A closed set $A$ is prox-regular at $\bar x\in A$ if the projector $P_A$ is single-valued near $\bar x$.
It follows (again, Lemma \ref{lem:pane.refectors}) that $\M$ is super-regular at a distance with arbitrarily
small violation as long as the neighborhood is correspondingly small.
The sparsity set $\SR$
has the most challenging structure.  We show in the proof of Lemma \ref{lem:pane.refectors}
that this set is $\epsilon$-super-regular at a distance, from which the regularity of the
projector and reflectors for this set follows.

By \cite[Proposition 3.1]{russell2018quantitative} we have the following implications:
\begin{equation*}
\begin{array}{ccc}
\text{convexity}&\Rightarrow      \text{prox-regularity} &\Rightarrow \epsilon\text{-super-regularity at a distance}.
\end{array}
\end{equation*}

\textit{Proof of Lemma \ref{l:SYMM}.}
The representation \eqref{e:symm as intersection} follows immediately from the definition of the set
$\SYM$ \eqref{eq:symmetry_set}.

It is easy to see that $\SYM^1_+,\SYM^2_-,\SYM^3_-$ are subspaces of the Hilbert space $\mathbb C^N$.
The formulas for the projections onto these subspaces, \eqref{e:P_symm1}-\eqref{e:P_symm3} follow from \cite[Section 9.7, exercise 2]{dennis1996numerical}.
Now, by \cite{Neumann50}, from any starting point $u$, the cyclic projections algorithm applied to subspaces converges to the projection of $u$ onto the intersection of the subspaces.  A routine calculation shows that $u^{(1)}:=P_{\SYM^3_-}\circ P_{\SYM^2_-}\circ P_{\SYM^1_+} u$ satisfies
$u^{(1)}:=P_{\SYM^3_-}\circ P_{\SYM^2_-}\circ P_{\SYM^1_+} u^{(1)}$, in other words, $u^{(1)}$ is a fixed point of the cyclic projections algorithm.  By
\cite{Neumann50} $u^{(1)}$ is therefore the projection of $u$ onto $\SYM$.
Hence, since $u$ is arbitrary,
$P_\SYM  = \left(P_{\SYM^3_-}\circ P_{\SYM^2_-}\circ P_{\SYM^1_+}\right)$ as claimed.\qedbox

\textit{Proof of Lemma \ref{lem:prox.regular.sets}.}
Part \eqref{lem:prox.regular.sets i}. Since$\M$ is determined by the equalities of continuous functions, it is closed.
Prox regularity of sets of the form \eqref{eq:phase_reformulated} was first used in \cite{Luke08} and follows from uniqueness of the projector $P_{\M}$ on small enough neighborhoods of $\M$.

Part \eqref{lem:prox.regular.sets ii}. It is a simple exercise to show that $\SYM$, $\SUPP$, and $\LF$ are convex.

Part \eqref{lem:prox.regular.sets iii}. Before we prove the third statement, we explain a little about the logic and technical difficulties.  The theory developed in
\cite{russell2018quantitative} determines the regularity of projectors and reflectors from the regularity of the sets.
To this point, we have not characterized the regularity of the set $\SR$, but we know that this is inherited by the
regularity of the sets $\mathcal A_s$ defined in \eqref{eq:sparse.set} and the real subspace $\mathscr{R}^N$ since
$\SR=\mathcal A_s\cap \mathscr{R}^N$.  Our proof of the regularity of $P_{\SR}$ and $R_{\SR}$ follows from the regularity of
$P_{\mathcal{A}_s}$ and $P_{\mathscr{R}^N}$ rather than using an explicit characterization of the regularity of the
set $\SR$.

While $\mathcal A_s$ is a subset of $\mathbb{C}^N$, this can be transformed isometrically to a subset of $\R^{2N}$,
where such sparsity sets have already been analyzed \cite{HesseLukeNeumann14}.
Let $u\in \SR\backslash\{0\}$.
Then we have $u\in \mathcal A_s\backslash \{0\}$ with $\mathcal A_s$ being defined in \eqref{eq:sparse.set}.
By \cite[Theorem III.4]{HesseLukeNeumann14}, $\mathcal A_s$ is $(0, \delta)$-subregular at $u$ for
$\delta\in (0,\min\{|u_i|:i\in I(u)\})$, where $I(u)=\{i\in\{1,\dots,2N\}:u_i\ne 0\}$ with the odd elements indexing
the real part of the vector $u\in \mathbb{C}^N$ and the even elements indexing the imaginary part. 
Using \cite[Proposition 3.1 (ii)]{russell2018quantitative} this implies that $\mathcal A_s$ is $\epsilon$-subregular at $u$
for all normal vectors $v\in N_{{\mathcal A}_s}(w)$ where $w\in \mathbb{B}_\delta(u)\cap {\mathcal A}_s$
with violation $\epsilon=0$, and hence $\epsilon$-super-regular at $u$ with violation $\epsilon=0$ \cite[Proposition 3.1 (v)]{russell2018quantitative}.
By \cite[Proposition 3.4(ii)]{luke2018relaxed}, there are a set $\Lambda\subset \mathbb C^N$ and a neighborhood $U\subset\mathbb C^N$ of $u$
such that the projector $P_{\mathcal A_s}$ is pointwise $\alpha$-fne at each $y'\in\Lambda$
with $\alpha=1/2$ and violation $0$ on $U$; that is,
\begin{eqnarray}\label{eq:panex}
&&\|x-y\|^2+\|(x'-x)-(y'-y)\|^2\leq  \|x'-y'\|^2\\
&&\forall x'\in U,\,y'\in\Lambda,\, x\in P_{\mathcal A_s} x',\,y\in P_{\mathcal A_s} y'.
\end{eqnarray}
Let $\check\Lambda=P_{\mathscr{R}^N}^{-1}(\Lambda)$ and $\check U=P_{\mathscr{R}^N}^{-1}(U)$.
Let $\check y'\in \check\Lambda$ and $\check x'\in \check U$.
Then $P_{\mathscr{R}^N}\check y'\in \Lambda$ and $P_{\mathscr{R}^N}\check x'\in U$.
By \eqref{e:P_SR} we have $P_{\SR}=P_{\sparse}P_{\mathscr{R}^N}$, so for all
$\check x\in P_{\SR}\check x'=P_{\sparse}P_{\mathscr{R}^N}\check x'$ and $\check y\in P_{\SR}\check y'= P_{\sparse}P_{\mathscr{R}^N}\check y'$, we have
\begin{equation}
\begin{array}{rl}
& \|\check x-\check y\|^2+\|(\check x'-\check x)-(\check y'-\check y)\|^2\\[5pt]
&\quad =\|\check x-\check y\|^2+\left(\|(\text{Re}(\check x'-\check x))-(\text{Re}(\check y'-\check y))\|^2+\|\text{Im}(\check x')-\text{Im}(\check y')\|^2\right)\\[5pt]
&\quad=\left(\|\check x-\check y\|^2+\|(P_{\mathscr{R}^N}\check x'-\check x)-(P_{\mathscr{R}^N}\check y'-\check y)\|^2\right)+\|\text{Im}(\check x')-\text{Im}(\check y')\|^2\\[5pt]
&\quad \leq  \|P_{\mathscr{R}^N}\check x'-P_{\mathscr{R}^N}\check y'\|^2+\|\text{Im}(\check x')-\text{Im}(\check y')\|^2\quad\text{(by \eqref{eq:panex})}\\[5pt]
&\quad \leq  \|\text{Re}(\check x')-\text{Re}(\check y')\|^2+\|\text{Im}(\check x')-\text{Im}(\check y')\|^2\\[5pt]
&\quad = \|\check x'-\check y'\|^2.
\end{array}
\end{equation}
This establishes that the projector $P_\SR$ is pointwise $\alpha$-fne at each $\check y'\in\check \Lambda$
($\alpha=1/2$ and violation $0$) on $\check U$.
By \cite[Proposition 2.3(ii)]{russell2018quantitative}, the reflector $R_\SR$ is therefore
pointwise nonexpansive at each $\check y'\in\check \Lambda$
(violation $0$) on $\check U$.  This establishes the third statement and completes the proof.
\qedbox

\begin{lemma}\label{lem:Tfirmly.nonep}
Let $\lambda\in [0,1]$ be fixed, and for $j=1,2$, let $B_j$ be a subset of $\mathbb E$ such that $R_{B_j}$ is pointwise almost nonexpansive
with violation $\tilde \epsilon_j$  at all points in $\Lambda_j$ on
$U_j$.
Suppose that $ R_{B_{2}}\Lambda_{2}\subseteq \Lambda_1$
and $R_{B_{2}}U_{2}\subseteq U_1$.
Then the two-set relaxed Douglas-Rachford mapping defined by
 \begin{equation}\label{eq:def.DRl}
     T=\frac{\lambda}{2}(R_{B_1}R_{B_{2}}+\Id)+(1-\lambda)P_{B_{2}}
 \end{equation}
is pointwise a$\alpha$-fne at all $y\in \Lambda_{2}$ with constant $\alpha=1/2$ and
violation $\tilde \epsilon$ on $U_{2}$
where
\begin{equation}
    \tilde \epsilon:=\frac{1}{2}\left[\left(\lambda\sqrt{1+\tilde\epsilon_{1}}+1-\lambda\right)^2(1+\tilde\epsilon_{2})-1\right].
\end{equation}
\end{lemma}
\begin{proof}
Let $y\in \Lambda_{2}$ and  $x\in U_{2}$.
By \cite[Proposition 2.1]{russell2018quantitative}, it is sufficient to prove that $\tilde{T}:= 2T-\Id$
is pointwise almost nonexpensive at $y$ with violation $2\tilde\epsilon$ on $U_{2}$.
Starting with \eqref{eq:def.DRl} we derive an equivalent representation for
$\tilde{T}$:
\begin{align}\label{e:TR/ST}
    \tilde{T} &=2\left(\frac{\lambda}{2}(R_{B_1} R_{B_{2}} +\Id)+(1-\lambda)P_{B_{2}}  \right)-\Id\\
&=\lambda (R_{B_1} R_{B_{2}} +\Id) +2(1-\lambda)P_{B_{2}} -\Id \\
&=\lambda R_{B_1} R_{B_{2}} + 2(1-\lambda)P_{B_{2}} + (\lambda-1)\Id\\
&=\lambda R_{B_1} R_{B_{2}} + (1-\lambda)(2P_{B_{2}}-\Id)\\
&=\lambda R_{B_1} R_{B_{2}} + (1-\lambda)R_{B_{2}}\\
&=\left(\lambda R_{B_1}+ (1-\lambda)\Id\right) R_{B_{2}}.
\end{align}
Let $y^+\in \tilde{T}y$ and $x^+\in \tilde{T}x$. We claim that
\begin{equation}
    \|y^+-x^+\|\leq\sqrt{1+2\tilde\epsilon}\|x-y\|.
\end{equation}
Indeed, let $u\in R_{B_{2}}y$, $z\in R_{B_{2}}x$, $u'\in R_{B_1}u$, $z'\in R_{B_1}z$ so
that by \eqref{e:TR/ST}
\begin{align*}
\tilde{T}y &= \left(\lambda R_{B_1}+ (1-\lambda)\Id\right) R_{B_{2}}y\ni
\lambda u' + (1-\lambda)u,\\
\tilde{T}x &= \left(\lambda R_{B_1}+ (1-\lambda)\Id\right) R_{B_{2}}x
\ni \lambda z' + (1-\lambda)z.
\end{align*}
By assumption, $u\in R_{B_{2}}\Lambda_{2}\subseteq \Lambda_1$
and $z\in R_{B_{2}}U_{2}\subseteq U_1$.
Since $R_{B_{2}}$
is pointwise almost nonexpansive with
violation $\tilde \epsilon_{2}$  at each point in $\Lambda_{2}$ on $U_{2}$, $R_{B_{2}}$ is also single-valued
at every point in $\Lambda_{2}$ so we can write $u = R_{B_{2}}y$,  and
\begin{subequations}\label{eq:nonexp}
\begin{equation}
    \|z-u\|\leq \sqrt{1+\tilde\epsilon_{2}}\|x-y\|.
\end{equation}
By the same argument, $R_{B_1}$ is pointwise almost nonexpansive (and hence single-valued)
with violation $\tilde \epsilon_1$ at all points in $\Lambda_1$ on
$U_1$. We can therefore write $u'=R_{B_1}u$, and
\begin{equation}
    \|z'-u'\|\leq \sqrt{1+\tilde\epsilon_1}\|z-u\|.
\end{equation}
\end{subequations}
Let $u'' = \lambda u'+(1-\lambda)u $ and $z'' = \lambda z'+(1-\lambda)z$.
Then $u''\in\tilde{T}y$ and $z''\in \tilde{T}x$. We have
\begin{align*}
    \|z''-u''\| &=\|\lambda z'+(1-\lambda)z-\lambda u'-(1-\lambda)u  \|\\
    &\leq \lambda\|z'-u'\| +(1-\lambda)\|z-u\|\\
    &\leq \lambda \sqrt{1+\tilde\epsilon_1}\|z-u\| +(1-\lambda)\|z-u\|\\
    &= \left(\lambda\sqrt{1+\tilde\epsilon_1}+1-\lambda\right)\|z-u\| \\
    &\leq \left(\lambda\sqrt{1+\tilde\epsilon_1}+1-\lambda\right)\sqrt{1+\tilde\epsilon_{2}}\|x-y\|\\
    &= \sqrt{1+2\tilde\epsilon}\|x-y\|.
\end{align*}
The first inequality follows from the triangle inequality.
The other  inequalities are a consequence of \eqref{eq:nonexp}.  
Since $y\in\Lambda_{2}$ and $x\in U_{2}$ are arbitrary, the result follows.
\end{proof}

\begin{lemma}(Pointwise almost nonexpansiveness of projectors/reflectors)\label{lem:pane.refectors}
\begin{enumerate}[(i)]
    \item\label{lem:pane.refectors i} Let  Assumption \ref{a:1} \eqref{a:1,b} and \eqref{a:1,d'} hold.
For $j=1,\dots,5$, $R_{A_{j}}$
is pointwise almost nonexpansive with
violation $\tilde \epsilon_{j}$  at each point in $\Lambda_{j}$ on $U_{j}$ with $\tilde \epsilon_{1}=\tilde\epsilon_{2}=\tilde\epsilon_{3}=\tilde \epsilon_{4}=0$ and $\tilde \epsilon_{5}=8\epsilon_{U_5}(1 + \epsilon_{U_5})/(1 -\epsilon_{U_5})^2$.
\item\label{lem:pane.refectors ii} Let  Assumption \ref{a:1} \eqref{a:1,b} and \eqref{a:1,f} hold.
For $j=1,\dots,5$, $P_{A_{j}}$
is pointwise a$\alpha$-fne with
violation $\check\epsilon_{j}$ and constant $\alpha_j=\frac{1}{2}$ at each point in $\Lambda_{j}$ on $U_{j}$ with $\check \epsilon_{1}=\check\epsilon_{2}=\check\epsilon_{3}=\check \epsilon_{4}=0$ and $\check \epsilon_{5}=4\epsilon_{U_5}(1 + \epsilon_{U_5})/(1 -\epsilon_{U_5})^2$.
\end{enumerate}
\end{lemma}
\begin{proof}
By Lemma \ref{lem:prox.regular.sets} \eqref{lem:prox.regular.sets ii}, the sets $\SYM$, $\SUPP$ and $\LF$ are convex, so they are everywhere $\epsilon$-super-regular at a distance with
$\epsilon=0$.
In other words, for $j=1,3,4$, $A_j$ is $\epsilon$-super-regular at a distance at any  $x_j$ relative to $\Lambda_j$
on any neighborhood $U_j$
 with constant $\epsilon_{U_j}=0$.

By Lemma \ref{lem:prox.regular.sets} \eqref{lem:prox.regular.sets i}, the set$\M$ is prox regular and hence there is a neighborhood where it is super-regular at a distance;
by assumption, $U_5$ is such a neighborhood.

For $j=1,3,4,5$,
set $\tilde\epsilon_j:= 8\epsilon_{U_j}(1 + \epsilon_{U_j})/(1 -\epsilon_{U_j})^2$ and $\check\epsilon_j:= 4\epsilon_{U_j}(1 + \epsilon_{U_j})/(1 -\epsilon_{U_j})^2$.
Then, $\tilde\epsilon_1=\tilde\epsilon_3=\tilde\epsilon_4=0$ and $\check\epsilon_1=\check\epsilon_3=\check\epsilon_4=0$ since $\epsilon_{U_j}=0$ for $j=1,3,4$.
The result follows from \cite[Proposition 3.4]{luke2020convergence} since $A_j$'s are  $\epsilon$-super-regular  at a distance of  at   $x_j$ relative to $\Lambda_j$
on the neighborhood $U_j$
 with constant $\epsilon_{U_j}$, for $j=1,3,4,5$.

For $j=2$, $A_2=\SR$ and by Lemma \ref{lem:prox.regular.sets} \eqref{lem:prox.regular.sets iii}, there is a neighborhood on which the reflector $R_{A_2}$ is
pointwise almost nonexpansive (with
violation $\tilde \epsilon_{2}=0$), and projector $P_{A_2}$ is pointwise $\alpha$-fne with $\alpha=1/2$, and $U_2$ together with $\Lambda_2$ are by assumption the sets where this holds.
\end{proof}

\begin{lemma}\label{t:T_orbit3D aalpha-fne}
Let $\lambda\in (0,1)$ and let $\Fix T^{\text{CDR$\lambda$}}\neq \emptyset$.
Under Assumptions \ref{a:1} \eqref{a:1,b}-\eqref{a:1,d'},
$T^{\text{CDR$\lambda$}}$ is pointwise a$\alpha$-fne at all $u\in \Fix T^{\text{CDR$\lambda$}}$
with
\begin{equation}\label{e:CDRl alpha-epsilon}
\overline\alpha=5/6\quad\mbox{and} \quad\overline\epsilon=\left(1+\epsilon_{4,5}\right)\left(1+\epsilon_{5,6}\right) - 1 \in [0,1),
\end{equation}
where
\begin{equation}
    \epsilon_{j,j+1}:=\begin{cases}
                       \tfrac12\tilde{\epsilon_5}&\mbox{ if }j=4,\\
                       \tfrac12\left[\left(\lambda\sqrt{1+\tilde\epsilon_5}+1-\lambda\right)^2-1 \right]&\mbox{ if } j=5,
                      \end{cases}
%
\end{equation}
for
$\tilde{\epsilon}_5= 8\epsilon_{U_5}(1 + \epsilon_{U_5})/(1 -\epsilon_{U_5})^2$ with
$\epsilon_{U_5}$ the constant of $\epsilon$-super-regularity at a distance of the set$\M$ on $U_5$ in the
range given by Assumption \ref{a:1}\eqref{a:1,d'}.
\end{lemma}
\begin{proof}
The proof is very similar to the proof of \cite[Proposition 4.1]{DinhJansenLuke2024CDRl}.  The difference here
is that the regularity of the set $\SR$ has not been explicitly determined, so the statement of the theorem
in the present setting relies on slightly different assumptions than the statement of
\cite[Proposition 4.1]{DinhJansenLuke2024CDRl}.  The logic is as follows:  we first establish that the
reflectors $R_{j}$ are almost nonexpansive for $j=1,2,\dots,5$;  next we show that each
$T_{j, j+1}$ in Assumption \ref{a:1} \eqref{a:1,d} is a$\alpha$-fne with $\alpha_{j,j+1}=1/2$ and
a specified violation  $\epsilon_{j,j+1}$.  We then use the calculus of a$\alpha$-fne mappings to conclude
that the composition of the $T_{j,j+1}$ mappings, namely $T^{\text{CDR$\lambda$}}$, is a$\alpha$-fne
with the claimed violation $\overline{\epsilon}$ and constant $\overline{\alpha}$.

By Lemma \ref{lem:pane.refectors}\eqref{lem:pane.refectors i} (which requires Assumption \ref{a:1}\eqref{a:1,b} and \eqref{a:1,d'}),
for $j=1,\dots,5$, $R_{A_{j}}$
is pointwise almost nonexpansive with
violation $\tilde \epsilon_{j}$  at each point in $\Lambda_{j}$ on $U_{j}$.

We claim that for $j=1,\dots,5$, the neighborhood $U_j$ and the set
 $\Lambda_j$ satisfy
$R_{A_{j+1}}U_{j+1}\subseteq U_{j}$ and $R_{A_{j+1}}\Lambda_{j+1}\subseteq \Lambda_{j}$.
By Assumption \ref{a:1} \eqref{a:1,c}, this holds for $j=2,5$.
For $j\in \{1,3,4\}$, the claim also holds since $U_{j}=\Lambda_j=\mathbb C^N$ in this case.

For $j=1,\dots,5$, set $\epsilon_{j,j+1}:=\frac{1}{2}\left[\left(\lambda\sqrt{1+\tilde\epsilon_{j}}+1-\lambda\right)^2(1+\tilde\epsilon_{j+1})-1\right]$.
If $j=4$, then $\epsilon_{4,5}=\frac{1}{2}\tilde \epsilon_5$, and
if $j=5$, then $\epsilon_{5,6}=\tfrac12\left[\left(\lambda\sqrt{1+\tilde\epsilon_5}+1-\lambda\right)^2-1 \right]$.
If $j=1,2,3$, then $\epsilon_{j,j+1}=0$.
By Lemma \ref{lem:Tfirmly.nonep}, the two-set relaxed Douglas-Rachford mapping $T_{j,j+1}$ defined in Assumption \ref{a:1} \eqref{a:1,d}
for $j=1,2,\dots,5$
is pointwise a$\alpha$-fne at all $y\in \Lambda_{j+1}$ with constant $\alpha_{j,j+1}=1/2$ and
violation $\epsilon_{j,j+1}$ on $U_{j+1}$.
This finishes the second stage of the proof.

For the final step of the proof, note that by \eqref{eq:CDRl_orbit3D}, the cyclic relaxed
Douglas-Rachford mapping $T^{\text{CDR$\lambda$}}$ is the composition of the $T_{j,j+1}$ mappings, i.e.,
$$T^{\text{CDR$\lambda$}}=T_{1,2}\circ T_{2,3}\circ T_{3,4}\circ T_{4,5}\circ T_{5,6},$$
where $T_{j,j+1}$
is pointwise a$\alpha$-fne at all $y\in \Lambda_{j+1}$ with constant $\alpha_{j,j+1}=1/2$ and
violation $\epsilon_{j,j+1}$ on $U_{j+1}$.
With $m=5$, Assumption \ref{a:1}\eqref{a:1,d} implies that  for $j=1,\dots,5$, $T_{j,j+1}\Lambda_{j+1}\subseteq \Lambda_{j}$ and
$T_{j,j+1}U_{j+1}\subseteq U_{j}$, which allows application of \cite[Proposition 4.1]{DinhJansenLuke2024CDRl} to conclude that
$T^{\text{CDR$\lambda$}}$ is pointwise
a$\alpha$-fne at all $y\in \Lambda_1=\mathbb C^N$ on $U_1=\mathbb C^N$ with
\begin{equation}\label{eq:violation.composite}
 \overline\alpha = \frac{m}{m+1}=\frac{5}{6}
\quad\mbox{ and violation }\quad
    \overline\epsilon=\prod_{j=1}^m(1+\epsilon_{j,j+1})-1=\left(1+\epsilon_{4,5}\right)\left(1+\epsilon_{5,6}\right) - 1.
\end{equation}
This completes the proof.
\end{proof}

\begin{lemma}\label{t:T_orbit3D CP aalpha-fne}
Let  $\Fix T^{\text{CP}}\neq \emptyset$.
Under Assumptions \ref{a:1} \eqref{a:1,b}, \eqref{a:1,e} and \eqref{a:1,f},
$T^{\text{CP}}$ is pointwise a$\alpha$-fne at all $u\in \Fix T^{\text{CP}}$
with
\begin{equation}\label{e:CP alpha-epsilon}
\overline\alpha=5/6\quad\mbox{and} \quad\overline\epsilon=\check\epsilon_{5}  \in [0,1),
\end{equation}
where
$\check{\epsilon}_5= 4\epsilon_{U_5}(1+\epsilon_{U_5})/(1-\epsilon_{U_5})^2$.
\end{lemma}
\begin{proof}
By Lemma \ref{lem:pane.refectors}\eqref{lem:pane.refectors ii} (which requires Assumption \ref{a:1}\eqref{a:1,b} and \eqref{a:1,f}), for $j=1,\dots,5$, $P_{A_j}$
is pointwise a$\alpha$-fne with
violation $\check\epsilon_{j}$ and constant $\alpha_j=\frac{1}{2}$ at each point in $\Lambda_{j}$ on $U_{j}$ with $\check \epsilon_{1}=\check\epsilon_{2}=\check\epsilon_{3}=\check \epsilon_{4}=0$ and $\check \epsilon_{5}=4\epsilon_{U_5}(1 + \epsilon_{U_5})/(1 -\epsilon_{U_5})^2$ where $\epsilon_{U_5}$ is in the range given by Assumption \ref{a:1}\eqref{a:1,f}.

By \eqref{eq:first.CP}, the cyclic projection mapping $T^{\text{CP}}$ is the composition of the $P_{A_j}$ mappings, i.e.,
$$T^{\text{CP}}=P_{A_1}\circ P_{A_2}\circ P_{A_3}\circ P_{A_4}\circ P_{A_5}.$$
With $m=5$, Assumption \ref{a:1}\eqref{a:1,e} implies that  for $j=1,2,3,4$, $P_{A_{j+1}}U_{j+1}\subseteq U_{j}$ and $P_{A_{j+1}}\Lambda_{j+1}\subseteq \Lambda_{j}$, which allows application of \cite[Proposition 2.4(iii)]{russell2018quantitative} to conclude that
$T^{\text{CP}}$ is pointwise
a$\alpha$-fne at all $y\in \Lambda_5$ on $U_5$ with
\begin{equation}\label{eq:violation.composite.CP}
 \overline\alpha = \frac{m}{m+1}=\frac{5}{6}
\quad\mbox{ and violation }\quad
    \overline\epsilon=\prod_{j=1}^m(1+\check\epsilon_{j})-1=\left(1+\check\epsilon_{5}\right) - 1=\check\epsilon_{5}.
\end{equation}
This completes the proof.
\end{proof}

\begin{lemma}\label{lem:pane.reflector.product}
Let $A_j$ be subset of $\mathbb E$ ($j=1,\dots,m$) such that $R_{A_j}$ is pointwise almost nonexpansive
with violation $\tilde \epsilon_j$  at all points in $\Lambda_j$ on
$U_j$. Set $\Lambda:={\Lambda_1\times \dots \times \Lambda_m}$ and $U:={U_1\times \dots \times U_m}$.
Then $R_{A_1\times \dots \times A_m}$ is pointwise almost nonexpansive
with violation $\tilde \epsilon=\max\limits_{j=1,\dots,m}\tilde \epsilon_j$  at all points in $\Lambda$ on
$U$.
\end{lemma}
\begin{proof}
Let $\mathbf{x}=(x_1,\dots,x_m)\in\Ebb^m$. We have
\begin{equation}
P_{A_1\times\dots\times A_m}\mathbf{x} = P_{A_1\times\dots \times A_m}(x_1,\dots,x_m)=P_{A_1}x_1\times\dots\times P_{A_m}x_m.
\end{equation}
It implies that
\begin{equation}\label{eq:reflector.product}
\begin{array}{rl}
     R_{A_1\times\dots\times A_m}\mathbf{x}& =2P_{A_1\times\dots\times A_m}\mathbf{x}-\mathbf{x}\\
     &= 2P_{A_1}x_1\times\dots\times P_{A_m}x_m-(x_1,\dots,x_m)\\
     &= (2P_{A_1}x_1-x_1)\times\dots\times (2P_{A_m}x_m-x_m)\\
     &=R_{A_1}x_1\times\dots\times R_{A_m}x_m.
\end{array}
\end{equation}
By assumption, we have
\begin{equation}\label{eq:panexp}
         \|x^+_j-y^+_j\|\leq\sqrt{1+\tilde\epsilon_j}\|x_j-y_j\|,\ \quad \ \forall x_j\in U_j,\  x^+_j\in R_{A_j}x_j,\ \forall y_j\in \Lambda_j,\ y^+_j\in R_{A_j}y_j.
\end{equation}
Let $\mathbf{x}=(x_1,\dots,x_m)\in U=U_1\times\dots\times U_m$ and $\mathbf{y}=(y_1,\dots,y_m)\in \Lambda=\Lambda_1\times\dots\times \Lambda_m$. Let $\mathbf{x^+}\in R_{A_1\times\dots\times A_m}\mathbf x$ and $\mathbf{y^+}\in R_{A_1\times\dots\times A_m}\mathbf y$. Then, by \eqref{eq:reflector.product}, $\mathbf{x^+}=(x^+_1,\dots,x^+_m)$ and $\mathbf{y^+}=(y^+_1,\dots,y^+_m)$ with $x^+_j\in R_{A_j}x_j$ and $y^+_j\in R_{A_j}y_j$.
We have
\begin{equation}
    \begin{array}{rl}
        \|\mathbf{x^+}-\mathbf{y^+}\| &=\sqrt{\sum_{j=1}^{m}\|x^+_j-y^+_j\|^2}  \\
         &\leq \sqrt{\sum_{j=1}^{m}(1+\tilde\epsilon_j)\|x_j-y_j\|^2}, \quad\quad \mbox{ (by \eqref{eq:panexp})}\\
         &\leq \sqrt{\max\limits_{j=1,\dots,m}{(1+\tilde\epsilon_j)}\sum_{j=1}^{m}\|x_j-y_j\|^2}\\
         &=\sqrt{(1+\max\limits_{j=1,\dots,m}\tilde\epsilon_j)\sum_{j=1}^{m}\|x_j-y_j\|^2}\\
         &= \sqrt{(1+\tilde \epsilon)}\sqrt{\sum_{j=1}^{m}\|x_j-y_j\|^2}\\
         &= \sqrt{(1+\tilde \epsilon)}\|\mathbf{x}-\mathbf{y}\|.
    \end{array}
\end{equation}
This concludes the statement.
\end{proof}

\textit{Proof of Theorem \ref{t:CDRl Orbital convergence}.}
 By Lemma \ref{t:T_orbit3D aalpha-fne}, under Assumption \ref{a:1} and Assumption \ref{ass:regularity}\eqref{t:msr convergence c}
 the mapping
 $T^{\text{CDR}\lambda}$ is pointwise a$\alpha$-fne at all $y\in \Fix T^{\text{CDR}\lambda}\neq\emptyset$
 with $\overline\alpha=5/6$ and $\overline\epsilon$ given by \eqref{e:CDRl alpha-epsilon},
 hence Assumption \ref{ass:regularity}\eqref{t:msr convergence a} is satisfied which yields \eqref{e:iterate bound}.
Convergence with the claimed rate
 follows from  Assumption \ref{a:msr convergence}  and Proposition \ref{t:metric.subreg.convergence}.
\qedbox

\textit{Proof of Theorem \ref{t:CP Orbital convergence}.}
    The proof is  similar to the one of Theorem \ref{t:CDRl Orbital convergence}, using Lemma \ref{t:T_orbit3D CP aalpha-fne}.
      \qedbox

In the following, we provide a proof of rates of convergence for the relaxed Douglas-Rachford algorithm on the product space specialized to the problem at hand.
Our approach is  simpler than the proof in \cite{luke2020convergence} and similar to the one used for cyclic projection and cyclic relaxed Douglas-Rachford above.

\textit{Proof of Theorem \ref{t:DRl Orbital convergence}.}
By Lemma \ref{lem:pane.refectors}\eqref{lem:pane.refectors i} (which requires Assumption \ref{a:1}\eqref{a:1,b} and \eqref{a:1,d'}), for $j=1,\dots,5$, $R_{A_{j}}$
is pointwise almost nonexpansive with
violation $\tilde \epsilon_{j}$  at each point in $\Lambda_{j}$ on $U_{j}$.
From Lemma \ref{lem:pane.reflector.product}, it follows that
$R_{C}$ with $C=A_1\times \dots \times A_m$ is pointwise almost nonexpansive
with violation $\tilde \epsilon_C=\max\limits_{j=1,\dots,m}\tilde \epsilon_j$  at all points in $\Lambda_C={\Lambda_1\times \dots \times \Lambda_m}$ on
$U_C={U_1\times \dots \times U_m}$.
Note that $\tilde \epsilon_C=\tilde \epsilon_5$ since for $j=1,2,3,4$, $\tilde \epsilon_j=0$.

Since the diagonal set $D$ is convex,  it is everywhere $\epsilon$-super-regular at a distance with
$\epsilon=0$.
In other words, $D$ is $\epsilon$-super-regular at a distance at any  $x$ relative to $\Lambda_D=\mathbb C^N$
on any neighborhood $U_D=\mathbb C^N$
 with constant $\epsilon_{U_D}=0$.
Thus $R_D$ is pointwise almost nonexpansive
with violation $\tilde \epsilon_D=0$ at all points in $\Lambda_D$ on
$U_D$.

By Lemma \ref{lem:Tfirmly.nonep}, $T^{\textbf{DR}\lambda}$, defined by \eqref{eq:DRl product},
is pointwise a$\alpha$-fne at all $y\in \Lambda_{D}$ with constant $\overline\alpha=1/2$ and
violation $\overline \epsilon$ given by
\begin{equation}
    \overline \epsilon:=\frac{1}{2}\left[\left(\lambda\sqrt{1+\tilde\epsilon_{C}}+1-\lambda\right)^2(1+\tilde\epsilon_{D})-1\right],
\end{equation}
on $U_{D}$, hence Assumption \ref{ass:regularity}\eqref{t:msr convergence a} is satisfied which yields \eqref{e:iterate bound}.
Convergence with the claimed rate
 follows from  Assumption \ref{a:msr convergence}  and Proposition \ref{t:metric.subreg.convergence}.
  \qedbox

\section*{Declarations}
\begin{itemize}
\item \textbf{Funding:}  All authors were supported in part by the German Research Foundation grant SFB1456.
 \item \textbf{Competing interests:} The authors have no financial or proprietary interests in any material discussed in this article.

\item \textbf{Author Contributions:}  Theoretical results were proved and presented by Russell Luke and Thi Lan Dinh.  Numerical models were developed by Matthijs Jansen and Russell Luke.  Numerical results with simulated data were obtained by Thi Lan Dinh and Matthijs Jansen.  Experimental data from the laboratory of Stefan Mathias was obtained by Wiebke Bennecke under the supervision of Stefan Mathias.  Numerical results on the experimental data were obtain by Thi Lan Dinh, with support from Wiebke Bennecke and Matthijs Jansen.
\item \textbf{Data availability:}  All data and code used to produce the results reported here are publicly available a \cite{Bennecke25_data}.
\end{itemize}

\end{document}